\documentclass[final]{siamltex}

\usepackage{times}
\usepackage{epsfig}
\usepackage{graphicx}
\usepackage{amsmath}
\usepackage{amssymb}
\usepackage{mathtools}
\usepackage{dsfont}
\usepackage{xspace}
\usepackage{mathrsfs}

\usepackage[pagebackref=true,breaklinks=true,letterpaper=true,colorlinks,bookmarks=false]{hyperref}

\newcommand{\beq}{\begin{equation*}}
\newcommand{\eeq}{\end{equation*}}
\newcommand{\beqn}{\begin{equation}}
\newcommand{\eeqn}{\end{equation}}
\newcommand{\beqa}{\begin{eqnarray*}}
\newcommand{\eeqa}{\end{eqnarray*}}
\newcommand{\beqan}{\begin{eqnarray}}
\newcommand{\eeqan}{\end{eqnarray}}

\newcommand{\notinclude}[1]{}
\newcommand{\comment}[1]{}
\newcommand{\missing}[1]{}

\DeclareMathOperator{\argmin}{{\mathrm{argmin}}}

\newcommand{\shell}{{\mathcal{S}}}

\newcommand{\manifold}{{\mathcal{M}}} 

\newcommand{\x}{{y}} 
\newcommand{\X}{{\mathbf{Y}}}
\newcommand{\V}{{\mathbf{V}}}
\newcommand{\LMg}{\mathscr{G}}
\newcommand{\Riesz}{\mathscr{R}}

\renewcommand{\grad}{{\mathop{\mathrm{grad}}}}
\newcommand{\length}{{\mathcal{L}}}  
\newcommand{\energy}{{\mathcal{E}}}  
\renewcommand{\L}{{\mathrm{L}}} 
\newcommand{\E}{{\mathrm{E}}} 
\newcommand{\Log}[1]{{(\tfrac1{#1}\mathrm{LOG})}}
\newcommand{\Exp}[1]{{(\mathrm{EXP}^{#1})}}
\newcommand{\ExpO}[2]{{\mathrm{EXP}_{#2}^{#1}}}
\newcommand{\parTp}{{\mathrm{P}}} 
\newcommand{\ParTp}{{\mathbf{P}}} 
\newcommand{\Nabla}{{\boldsymbol\nabla}}

\newcommand{\W}{{\mathcal{W}}}

\newcommand{\compDom}{{D}}
\newcommand{\setof}[2]{{\{{#1}\,|\,{#2}\}}}
\newcommand{\tgl}{{\mathrm{tgl}}}
\newcommand{\bnd}{{\mathrm{bnd}}}

\newcommand{\R}{{\mathds{R}}}
\newcommand{\N}{{\mathds{N}}}
\renewcommand{\d}{{\,\mathrm{d}}}

\newcommand{\dist}{{\mathrm{dist}}}

\newcommand{\Id}{{\mathds{1}}}
\newcommand{\Itau}{{{\mathcal{I}}_\tau}}

\newtheorem{remark}[theorem]{Remark}

\makeatletter
\DeclareRobustCommand\onedot{\futurelet\@let@token\@onedot}
\def\@onedot{\ifx\@let@token.\else.\null\fi\xspace}

\def\eg{\emph{e.g}\onedot} 
\def\ie{\emph{i.e}\onedot} 
\def\cf{\emph{cf}\onedot}

\def\etal{\emph{et al}\onedot}
\makeatother

\graphicspath{{./figuresSINUM/}}

\begin{document}

\title{Variational time discretization of geodesic calculus}

\author{Martin Rumpf\thanks{Bonn University, Endenicher Allee 60, D-53115 Bonn, Germany ({\tt Martin.Rumpf@ins.uni-bonn.de}).}
\and Benedikt Wirth\thanks{Courant Institute, \notinclude{New York University, }251 Mercer Street, New York NY 10012, USA ({\tt Benedikt.Wirth@cims.nyu.edu}).}}

\maketitle

\begin{abstract}
We analyze a variational time discretization of geodesic calculus on finite- and certain classes of infinite-dimensional Riemannian manifolds.
We investigate the fundamental properties of discrete geodesics, the associated discrete logarithm,
discrete exponential maps, and discrete parallel transport,
and we prove convergence to their continuous counterparts.
The presented analysis is based on the direct methods in the calculus of variation, on $\Gamma$-convergence,
and on weighted finite element error estimation. 
The convergence results of the discrete geodesic calculus are experimentally confirmed
for a basic model on a two-dimensional Riemannian manifold.
This provides a theoretical basis for the application to shape spaces in computer vision, for which we 
present one specific example.
\end{abstract}



\pagestyle{myheadings}
\thispagestyle{plain}
\markboth{MARTIN RUMPF AND BENEDIKT WIRTH}{VARIATIONAL TIME DISCRETIZATION OF GEODESIC CALCULUS}

\section{Introduction}
Riemannian geometry is a powerful theory which allows to transfer many important concepts
(shortest connecting paths, the arithmetic mean, or the principal component analysis, to name but a few examples)
from linear vector spaces onto nonlinear, curved spaces.
It is based on the notion of a Riemannian metric, and it provides a set of basic tools useful in applications,
among which we shall concentrate on geodesics, geodesic distance, exponential and logarithmic maps, and parallel transport.
During the past decade, Riemannian concepts have for instance increasingly been applied in computer vision
for the design and investigation of nonlinear and often infinite-dimensional shape spaces, where the Riemannian metric encodes the preferred shape variability.
Applications include shape morphing and modeling \cite{KiMiPo07},
computational anatomy in which the morphing path establishes correspondences between a patient and a template \cite{BeMiTrYo02}, as well as shape statistics \cite{FlLuPi04}.

Notwithstanding the conceptual power,
operators such as the logarithm or the exponential map involve the solution of time-dependent nonlinear ordinary or partial differential equations. In more complex spaces---as they appear for instance in vision applications---they are typically difficult to compute unless the spaces possess very peculiar structures \cite{YoMiSh08,SuMeSo11}.
Therefore, we here develop a discrete theory that can be seen as a natural time-discretization of the above-mentioned Riemannian calculus and which is comparatively simple to state and to implement.
It is centered around the definition of a discrete geodesic as the minimizer of a time-discrete path energy
and naturally extends from there to discrete analogs of logarithm, exponential map, and parallel transport.

Complementarily to earlier work \cite{RuWi12,HeRuWa12}, which deals with two different specific shape spaces and 
focuses on the experimental study of the resulting discrete calculus,
we here provide the theoretical justification of the general approach 
in terms of a rigorous convergence analysis for decreasing time step size under suitable 
assumptions on the manifold and the functional involved in the approximation.

The main intended application is a Riemannian calculus in the 
context of shape spaces in computer vision, which motivates our work.
There is a rich diversity of Riemannian shape spaces in the literature.
While some of them are finite-dimensional and represent shapes as polygonal curves or triangulated surfaces \cite{KiMiPo07,LiShDi10},
most deal with infinite-dimensional shapes such as planar curves 
with curvature-based or Sobolev-type metrics \cite{MiMu04,SrJaJo06}.
Dupuis\,\etal employed a higher order quadratic form on Eulerian 
motion velocity $g(v,v)=\int_\compDom L v\! \cdot\! v \,\d x$ on a 
computational domain $\compDom\subset\R^d$ to define a 
suitable metric on the space of diffeomorphisms on $\compDom$ \cite{DuGrMi98}.
An alternative approach uses the theory of optimal transport,
where image intensities are considered as probability densities \cite{ZhYaHa07} and
the Monge--Kantorovich functional $\int_{\!\compDom}\!|\psi(x)\!-\!x|^2\rho_0(x)\d x$ is minimized over all mass preserving mappings $\psi\!:\!\!\compDom\!\!\to\!\!\compDom$.
Via a flow reformulation due to Benamou and Brenier \cite{BeBr00}, this fits nicely into the Riemannian context.
There are only few nontrivial application-oriented Riemannian spaces in which geodesics can be computed in closed form (\eg \cite{YoMiSh08,SuMeSo11}),
else the system of geodesic ODEs has to be solved via time stepping (\eg \cite{KlSrMi04,BeMiTr05}).
Geodesics can also be obtained variationally by minimizing the discretized path length \cite{ScClCr06} 
or energy \cite{FuJuScYa09,WiBaRu10}, an approach generalized and simplified in this work.

The idea of variational time discretization, which underlies our discrete geodesic calculus,
has proved very appropriate also in other fields, in particular in the discretization of gradient flows
for an energy with respect to a particular (Riemannian) metric (\eg mean curvature flow \cite{AlTaWa93}).
For Hamiltonian mechanical systems, variational time discretization is by now a classic field.
The corresponding analog of our time-discrete path energy already occurred in the 70's as discrete action sum
(\cf the brief historic account in \cite{HaLuWa06}) and was analyzed as an independent, time-discrete system,
whose discrete symplectic time steps exhibit the same structure as the discrete exponential proposed here.
The consistency with the associated time-continuous mechanical systems was much later exploited numerically
\cite{HaLuWa06} and analyzed for particular systems from the $\Gamma$-convergence \cite{MuOr04}
and from the ODE-discretization perspective \cite{LeMaOr04} under the name of variational integrators.

Let us note that in contrast to the above works,
we here lay emphasis on the Riemannian tools, in particular on computing a discrete geodesic between two fixed given points,
on deriving a consistent discrete logarithm from this geodesic,
and on using discrete exponential and logarithmic map to compute a consistent discrete parallel transport.
In Section\,\ref{sec:discretecalculus} we briefly set forth the different components of our discrete geodesic calculus.
The existence and convergence properties of the discrete geodesics and corresponding operators are proved in Sections\,\ref{sec:gamma} and \ref{sec:convergence} after which we discuss how the theory applies to embedded manifolds and to a specific example of an infinite-dimensional manifold of shapes in Section\,\ref{sec:application}.

\section{Discrete geodesic calculus on a Riemannian manifold}
\label{sec:discretecalculus}
In this section we present the concept of the discrete geodesic calculus, 
which comprises the notions of discrete geo\-desics,
discrete logarithmic and exponential map, and discrete parallel transport. In preparation for this we briefly recall the corresponding continuous calculus. Precise assumptions on the manifold and involved approximating functionals will be given in the subsequent sections.

Let $(\manifold,g)$ be a smooth, complete Riemannian manifold, where $g$ denotes the metric,
given as a family of positive definite quadratic forms $g_\x:T_{\x}\manifold\times T_{\x}\manifold \to \R$
for $\x\in \manifold$, where $T_{\x}\manifold$ denotes the tangent space in $\x$. Given a smooth path $(\x(t))_{t\in [0,1]}$ on
$\manifold$, the length of this path is given by
\beqn\label{eq:contlength}
\length[(\x(t))_{t\in [0,1]}] = \int_0^1 \sqrt{g_{\x(t)}(\dot\x(t),\dot\x(t))} \d t\,.
\eeqn
Given two points $\x_A$ and $\x_B$ in $\manifold$, the minimizer of $\length$ over all smooth paths  $(\x(t))_{t\in [0,1]}$
with $\x(0) = \x_A$ and $\x(1) = \x_B$ is a geodesic. More generally, geodesics are defined as local minimizers of the path length.
Let us define the distance $\dist(\x_A,\x_B)$ between $\x(0) = \x_A$ and $\x(1) = \x_B$ as the minimal path length.
After reparameterization, a length-minimizing geodesic is also a minimizer of the path energy
\beqn\label{eq:contenergy}
\energy[(\x(t))_{t\in [0,1]}] = \int_0^1 g_{\x(t)}(\dot\x(t),\dot\x(t)) \d t
\eeqn
and satisfies the constant speed property $g_{\x(t)}(\dot\x,\dot\x) = \length^2[(\x(t))_{t\in [0,1]}]$.
The logarithm of a point $\x_B$ with respect to a point $\x_A$ is defined by
$\log_{\x_A} \x_B = \dot \x(0)$
if $(\x(t))_{t\in [0,1]}$ is the unique shortest path with constant speed connecting $\x(0) = \x_A$ and $\x(1) = \x_B$.
The exponential map $\exp_{\x_A}$ maps every tangent vector $v\in T_{\x_A}\manifold$ onto the endpoint
of a geodesic starting at $\x_A$ with initial speed $v$, \ie in the above notation,
$\exp_{\x_A} \dot \x(0) = \x(1)=\x_B\,.$
Furthermore, let us recall the parallel transport on a Riemannian manifold.
Using the Levi-Civita connection $\nabla$ a tangential vector field $v$ is parallel along a curve $(\x(t))_{t\in [0,1]}$ on $\manifold$ if
$\nabla_{\dot \x(t)} v(t)=0\,.$
If we sample a continuous path $(\x(t))_{t\in [0,1]}$ at times $t_k=k \tau$ for $k=0 ,\ldots, K$  and $\tau := \frac1K$, denoting
$\x_k=\x(t_k)$, we obtain the following estimates for length and energy
\begin{equation}\label{eq:contEstimates}
\length[(\x(t))_{t\in[0,1]}] \geq \sum_{k=1}^K\dist(\x_{k-1},\x_k)\,, \qquad
\energy[(\x(t))_{t\in[0,1]}] \geq \frac1\tau\sum_{k=1}^K\dist^2(\x_{k-1},\x_k)\,,
\end{equation}
where equality holds for geodesic paths due to the constant speed property. Indeed, the first estimate is straightforward, and application of Jensen's inequality yields the second estimate,
\beq
\sum_{k=1}^K\dist^2(\x_{k-1},\x_k)
\leq \sum_{k=1}^K \tau \; \int_{(k-1)\tau}^{k\tau} g_{\x(t)}(\dot\x(t),\dot\x(t))\, \d t = \tau \; \energy[(\x(t))_{t\in[0,1]}]\,.
\eeq

Based on these preliminaries we are now in the position to introduce a discrete geodesic calculus.
The starting point is a local approximation of the squared Riemannian distance $\dist^2$ by a functional $\W$.
In detail, we suppose that a smooth functional $\W: \manifold \times \manifold \to \R$ is given such that
\begin{equation}
\label{eq:distapprox}
\dist^2(\x,\tilde \x) = \W[\x,\tilde\x] + O(\dist^3(\x,\tilde \x))
\end{equation}
for all $\x,\, \tilde \x \in \manifold$. We will see later that $g_\x = \frac12 \W_{,22}[\x,\x]$ implies \eqref{eq:distapprox} for smooth $g$ and $\W$.
Together with \eqref{eq:contEstimates}  this motivates the following definition of a
\emph{discrete path energy} and a \emph{discrete path length} for discrete $K$-paths defined as
$(K+1)$-tuples $(\x_0,\ldots, \x_K)$ with $\x_k\in \manifold$ for $k=0,\ldots, K$.

\begin{definition}[Discrete length and energy]\label{def:discreteLE}
For a discrete $K$-path $(\x_0,\ldots,\x_K)$ we define the \emph{discrete length} $\L$ and the \emph{discrete energy} $\E$ by
\beq
\L[(\x_0,\ldots,\x_K)]=\sum_{k=1}^K \sqrt{\W[\x_{k-1},\x_k]}\,,\qquad 
\E[(\x_0,\ldots,\x_K)] =K\sum_{k=1}^K \W[\x_{k-1},\x_k]\,. \label{eq:discretePathEnergy}
\eeq
Then a \emph{discrete geodesic} (of order $K$) is defined as a minimizer of $\E[(\x_0,\ldots,\x_K)]$ for fixed end points
$\x_0,\x_K$.
\end{definition}

To proceed with the definition of discrete logarithm and discrete exponential map we need to give a meaning to displacements
$\tilde \x- \x$. Thus, in what follows let us assume that $\manifold$ is a subset of some Banach space $\X$.
This assumption is not restrictive because our constructions are purely local and can thus always be applied to charts of a more general manifold.
Given a continuous constant speed geodesic $(\x(t))_{t\in [0,1]}$ with $\x(0) = \x_A$ and $\x(1)= \x_B$ and a discrete geodesic $(\x_0,\ldots,\x_K)$ with $\x_0=\x_A$ and $\x_K=\x_B$,
we may view $\x_1-\x_0$ as the discrete counterpart of $\tau\dot\x(0)$ for $\tau=\frac1K$.
Motivated by the fact that $\frac1K \log_{\x_A}(\x_B)= \tau \dot \x(0)$ we hence give the following definition of a discrete logarithmic map.

\begin{definition}[Discrete logarithm]\label{def:discreteLog}
Suppose the discrete geodesic $(\x_0,\ldots,\x_K)$ is the
unique minimizer of the discrete energy \eqref{eq:discretePathEnergy} with $\x_0=\x_A$ and $\x_K=\x_B$, then
we define the \emph{discrete logarithm}
$
\Log{K}_{\x_A}(\x_B) = \x_1-\x_0\,.
$
Note that $\frac1K$ is part of the symbol and not a factor.
\end{definition}

In the special case $K=1$ we obtain $\Log{K}_{\x_A}(\x_B) = \x_B-\x_A$.
As in the continuous case, the discrete logarithm can be considered as a \emph{linear}
representation (in the space $\X$ of displacements on $\x_A$) of the \emph{nonlinear} variation $\x_B$ of $\x_A$.
As we will verify later under suitable assumptions, for a sequence of successively refined discrete geodesics we
obtain
$K \textstyle\Log{K}_\x(\tilde\x) \to \log_\x(\tilde\x)$
for $K\to \infty$.

In the continuous setting, the exponential map of a tangent vector $v \in T_{\x_A}\manifold$ is defined via $\exp_{\x_A} v = \x_B$, where $\x_B=\x(1)$ for the (unique) geodesic $(\x(t))_{t\in[0,1]}$ with $\x(0)=\x_A$ and $\dot \x(0)=v$.
Obviously, $\exp_{\x_A}(\frac{k}{K} v) = \x(\frac{k}{K})$ holds for $k=0,\ldots, K$.
We now aim at approximating $\exp_{\x_A}(k\cdot)$ via a discrete counterpart  $\ExpO{k}{\x_A}$ (the notation reflects the fact that on Lie-groups $\exp^k(\cdot) = \exp(k\cdot)$).
The definition should be consistent with the discrete logarithm in the sense that $\ExpO{k}{\x_0}(\zeta_1)=\x_k$
for a discrete geodesic $(\x_0,\x_1,\ldots,\x_K)$ of order $K\geq k$ with $\zeta_1 = \x_1-\x_0 = \Log{K}_{\x_0}(\x_K)$.
Our definition will reflect the following recursive properties of the continuous exponential map for $v \in T_{\x}\manifold$ sufficiently small,
\begin{align*}
\exp_\x(1v) &= \left(\textstyle\frac11\log_\x\right)^{-1}(v)\,, \quad 
\exp_\x(2v) = \left(\textstyle\frac12 \log_\x\right)^{-1}(v)
\notinclude{\Leftrightarrow\textstyle\frac12 \log_\x (\exp_\x (2v)) = v}\,, \\
\exp_\x(kv) &= \exp_{\exp_\x((k-2)v)}(2 v_{k-1}) 
\quad\text{for }v_{k-1} = \log_{\exp_\x ((k-2) v)} \exp_\x((k-1) v)\,.
\end{align*}
Replacing the tangent vector $v$ by a displacement $\zeta$, $\exp(k\cdot)$ by $\Exp{k}$, and
$\frac1k \log$ by $\Log{k}$ we obtain the following recursive definition of the discrete exponential map.

\begin{definition}[Discrete exponential map]
For $\zeta \in\X$ sufficiently small with $\x+\zeta\in\manifold$ define
\begin{align*}
\ExpO{1}{\x}(\zeta) &= \textstyle\Log{1}_\x^{-1}(\zeta)\,, \quad
\ExpO{2}{\x}(\zeta) = \textstyle\Log{2}_\x^{-1}(\zeta)\,,\\
\ExpO{k}{\x}(\zeta) &= \ExpO{2}{\ExpO{k-2}{\x}(\zeta)}(\zeta_{k-1}) 
\quad\text{with }\zeta_{k-1} = \textstyle\Log{1}_{\ExpO{k-2}{\x}(\zeta)} \ExpO{k-1}{\x}(\zeta)\,.
\end{align*}
\end{definition}

The smallness assumption on $\zeta$ ensures that $\Log{1}_\x$ and $\Log{2}_\x$ are invertible.
It is straightforward to verify $\ExpO{K}{\x}=\Log{K}_\x^{-1}$ on the image of $\Log{K}_\x$.
The central ingredient of this definition, the operator $\ExpO{2}{\x}$ requires the solution of a problem with a variational constraint. Indeed, to compute  $\ExpO{2}{\x}(\zeta)$, one considers all discrete geodesic paths $(\x, \x_1, \x_2)$ of order $2$,
where for any chosen $\x_2$ the point $\x_1$ must by definition be the unique minimizer of \eqref{eq:discretePathEnergy}
so that we may write $\x_1[\x_2]$. Now, $\ExpO{2}{\x} (\zeta)$ is that point $\x_2$ for which we have
$\x+\zeta = \x_1[\x_2]$, which can be equivalently described via the equation
\begin{equation}\label{eqn:variationalProblemExp2}
\x + \zeta = \argmin\limits_{\x_1\in \manifold} \left(\W[\x, \x_1] + \W[\x_1, \x_2]\right)\,.
\end{equation}
The corresponding Euler--Lagrange equation, as a nonlinear equation for
$\ExpO{2}{\x} (\zeta)=\x_2$, is given by
$
\W_{,2}[\x, \x+\zeta] + \W_{,1}[\x+\zeta, \x_2] = 0\,.
$
\begin{remark}[Discrete geodesics versus discrete exponential shooting]
The variational definition \eqref{eqn:variationalProblemExp2} of the operator $\ExpO{2}{}$ 
ensures that as long as a discrete geodesic connecting two end points $\x_A$, $\x_B$ is unique,
the discrete exponential map $\ExpO{k}{\x_A}$ applied to $\Log{K}_{\x_A}(\x_B)$ 
for $k=0,\ldots, K$ retrieves exactly the same discrete geodesic as the minimization
of the discrete energy $\E[\x_A, \hat \x_1, \ldots, \hat  \x_{K-1}, \x_B]$ over all $\hat  \x_1, \ldots, \hat  \x_{K-1}$, \ie 
$\x_k=\ExpO{k}{\x} \Log{K}_{\x_A}(\x_B)\,,$
where $(\x_A=x_0, \x_1, \ldots, \x_{K-1}, \x_B= \x_K)$ is the discrete geodesic.
\end{remark}

Finally, we consider the discretization of parallel transport along a discrete (not necessarily geodesic) curve $(\x_0,\ldots, \x_K)$. In the continuous setting, given a curve $(\x(t))_{t \in [0,1]}$ and a vector $v_0$ at $\x(0)$,
we define the parallel transport $\parTp_{(\x(t))_{t \in [0,1]}} v_0$ of $v_0$ along the curve as the vector $v(1)$ resulting from the solution of the $\nabla_{\dot\x(t)} v(t) =0$ for $t\in[0,1]$ and initial data $v(0) = v_0$, where $\nabla_{\dot\x}$ denotes the covariant derivative defined via the Levi-Civita connection $\nabla$.
There is a well-known first-order approximation of parallel transport called Schild's ladder \cite{EhPiSc72,KhMiNe00},
which is based on the construction of a sequence of geodesic parallelograms.
Given a curve $(\x(t))_{t \in [0,1]}$ and a tangent vector $v_{k-1}\in T_{\x((k-1)\tau)}\manifold$,
the approximation $v_k\in T_{\x(k \tau)}\manifold$ of the parallel transported vector via a geodesic parallelogram can be expressed as
\begin{align*}
\x^p_{k-1} &= \exp_{\x((k-1)\tau)} (v_{k-1}) \,, \quad
&\x_k^c &= \exp_{\x^p_{k-1}} \big(\tfrac12 \log_{\x^p_{k-1}}  (\x(k \tau)) \big)\,,\\
\x_{k}^p &= \exp_{\x((k-1)\tau)} \big(2 \log_{\x((k-1)\tau)}  (\x_k^c)\big) \,,\quad
&v_k &= \log_{\x(k \tau)} (\x_{k}^p)\,.
\end{align*}
Here, $\x_k^c$ is the midpoint of the two diagonals of the geodesic parallogramm with vertices $\x((k-1)\tau)$,
$\x^p_{k-1}$, $\x^p_k$, and $\x(k\tau)$.
This scheme can be easily transferred to discrete curves based on the discrete logarithm and the discrete exponential introduced before.

\begin{definition}[Discrete parallel transport]
Let $(\x_0,\ldots, \x_K)$ be a discrete curve in $\manifold$ with $\x_k-\x_{k-1}$ sufficiently small for $k=1,\ldots, K$,
then the \emph{discrete parallel transport} of a displacement $\zeta_0$ at $\x_0$ along $(\x_0,\ldots, \x_K)$
is defined for $k=1,\ldots, K$ via the iteration
\begin{align*}
\x^p_{k-1} &= \x_{k-1} + \zeta_{k-1}\,, \quad
&\x_k^c &= \x^p_{k-1} + \big( \Log{2}_{\x^p_{k-1}} (\x_{k})\big)\,, \\
\x_k^p &= \ExpO{2}{\x_{k-1}} \left(\x_k^c-\x_{k-1}  \right)\,, \quad
&\zeta_{k} &= \x_k^p -\x_k\,,
\end{align*}
where $\zeta_{k}$ is the transported displacement at $\x_k$. We denote
$\ParTp_{\x_K,\ldots,\x_0} \zeta_0 = \zeta_K$.
\end{definition}

Above we have used $\Log{1}_\x (\tilde \x) = \tilde \x - \x$  as well as $\ExpO{1}{\x}(\zeta) = \x+ \zeta$.
In the $k$\textsuperscript{th} step of the discrete parallel transport the Euler--Lagrange equations to determine $\x_k^c$ and $\zeta_k$ for given $\zeta_{k-1}$ and discrete curve $(\x_0,\ldots, \x_K)$ are
\beqan\label{eq:ELParallel1}
\W_{,2}[\x_{k-1}+\zeta_{k-1},\x_k^c] + \W_{,1}[\x_k^c, \x_{k}] &=& 0\,, \\
\label{eq:ELParallel2}
\W_{,2}[\x_{k-1},\x_k^c] + \W_{,1}[\x_k^c, \x_{k}+\zeta_{k}] &=& 0\,.
\eeqan
If $\W$ is symmetric, these conditions are the same as the Euler--Lagrange equations for backwards parallel transport
so that $\ParTp_{\x_K,\ldots,\x_0}^{-1}=\ParTp_{\x_0,\ldots,\x_K}$,
which however is generally not true if $\W$ is not symmetric.
Note that for a continuous geodesic $\x(t)$ the velocity $\dot\x(t)$ at a time $t$
equals the initial velocity $\dot\x(0)$, parallel transported along the geodesic.
In analogy, for a discrete geodesic $(\x_0,\ldots,\x_K)$ we have $\x_{k+1}-\x_k=\ParTp_{\x_k,\ldots,\x_0}(\x_1-\x_0)$ for $k=0,\ldots,K-1$.

Given discrete parallel transport, the Levi-Civita connection $\nabla_\xi\eta$
for $\xi \in T_\x \manifold$ and a vector field $\eta$ in the tangent bundle $T\manifold$
can be also be approximated.

\begin{definition}
For $\x \in \manifold$, small enough $\xi \in \X$, and vectors $\eta_0$ attached to $\x$ and $\eta_1$ attached to  $\x+\xi$
(representing a discrete vector field),
$
\Nabla_\xi (\eta_0,\eta_1)=\ParTp_{\x+\xi,\x}^{-1}\eta_1-\eta_0
$
defines a discrete connection. 
\end{definition}

The inverse parallel transport satisfies the same Euler--Lagrange equations \eqref{eq:ELParallel1}-\eqref{eq:ELParallel2} as above
(only this time $\x_k^c$ and $\zeta_{k-1}$ are determined from $\zeta_k$).
In analogy to $\nabla_{\dot\x(t)}\dot\x(t)=0$ for a continuous geodesic 
a discrete geodesic $(\x_0,\ldots,\x_K)$ satisfies 
$\Nabla_{\Delta\x_k}(\Delta\x_{k},\Delta\x_{k+1}) = 0 $ for 
$\Delta\x_k=\x_{k+1}-\x_k$, $k=0,\ldots,K-2$.
If $\W$ is symmetric, we can also express the discrete connection as 
$\Nabla_\xi(\eta_0,\eta_1)=\ParTp_{\x,\x+\xi}\eta_1-\eta_0$.

\section{Discrete geodesic calculus on embedded finite-dimensional manifolds}\label{sec:embeddedManifold}
As a simple example and to further motivate our approach
let us briefly demonstrate the discrete geodesic calculus
for the simple case of an $(m-1)$-dimensional manifold $\manifold$ embedded in $\R^m$.
We consider the simple energy
$\W[\x, \tilde \x] = |\tilde \x - \x|^2$
which reflects the stored elastic energy in a spring spanned
between points $\x$ and $\tilde \x$ through the ambient space of $\manifold$ in $\R^m$.
Thus, the discrete path energy is given by
\beqa\textstyle
\E[(\x_0,\ldots, \x_K)] = K \sum_{k=1}^K |\x_{k}-\x_{k-1}|^2\,.
\eeqa
Now, we define the Lagrangian
$F(X,\Lambda) = \E[(\x_0,\ldots, \x_K)] - \Lambda \cdot D(X)\,,$
where $X=(\x_1,\ldots, \x_{K-1})$ are the actual point positions, $\Lambda=(\lambda_1,\ldots, \lambda_{K-1})$ the Lagrange multipliers, and $D(X) = (d_\manifold(\x_k))_{k=1,\ldots, k-1}$ with $d_\manifold(\cdot)$ the (local) signed distance function of $\manifold$. Hence
$\grad_X F(X,\Lambda) =0$, $\grad_\Lambda F(X,\Lambda) =0$  are the necessary conditions
for $(\x_0,\ldots, \x_K)$ to be a discrete geodesic with end points $\x_0=\x_A$ and $\x_K=\x_B$.
The resulting system of $(m+1)(K-1)$ degrees of freedom can be solved for instance by Newton's method.
In particular, $(\grad_X F(X,\Lambda))_k = 2K\, (2 \x_k- \x_{k-1}-\x_{k+1}) - \lambda_k \nabla d_\manifold(\x_k)= 0$ results in finding points $\x_k$ ($k=1,\ldots, K-1$) on $\manifold$ such that
\beqn
2 \x_k- \x_{k-1}-\x_{k+1} \perp T_{\x_k} \manifold\,,
\label{eq:logspring}
\eeqn
while $\grad_\Lambda F(X,\Lambda) = D(X) = 0$ ensures that the points $\x_k$ are located on $\manifold$.
Equation \eqref{eq:logspring} is a discrete counterpart
of the continuous geodesic equation $\ddot{\x}(t) \perp T_{\x(t)} \manifold$ for an arclength parametrized geodesic curve $(\x(t))_{t\in \R}$ with $\frac{2 \x_k- \x_{k-1}-\x_{k+1}}{\tau^2} \approx - \ddot{\x}(k\tau)$.

From these considerations, we derive for the discrete logarithm that
$\zeta = \Log{2}_\x \tilde \x$ is described by the condition
\beqa\textstyle
 \zeta - \frac{\tilde\x - \x}{2} \perp T_{\x+\zeta} \manifold\,.
\eeqa
To achieve a simple geometric condition for the discrete exponential map $\ExpO{2}{\x} \zeta$ we consider a
discrete geodesic $(\x, \x_1, \x_2)$ with $\x_1 = \x  + \zeta$ and $\x_2 = \x_1 + \eta $, where the displacement $\eta$ is the actual degree of freedom.
The necessary condition for $(\x, \x_1, \x_2)$ to be a discrete geodesic and thus for $\x_2= \ExpO{2}{\x} \zeta$ is
$$
\zeta- \eta \perp T_{\x+\zeta} \manifold\,,
$$
which results in a one-dimensional search problem in the space $\manifold \cap \left(\x + \mbox{span} \{\zeta, n_{\x+\zeta} \} \right)$,
where $n_{\x+\zeta} $ denotes a normal on $T_{\x+\zeta} \manifold$.

\section{Properties of the discrete path energy}
\label{sec:gamma}
In the remainder of the paper we aim to examine the convergence properties of the discrete geodesic calculus
as the discrete time step size $\tau=\tfrac1K$ tends to $0$.
Figure\,\ref{fig:convergence} shows experimental evidence for the convergence
of the discrete geodesic, exponential map, logarithmic map, and parallel transport.
\begin{figure}
\setlength{\unitlength}{.25\linewidth}
\begin{picture}(4,1.9)
\put(-.02,-0.1){\includegraphics[trim=60  10  440 148,clip,width=1\unitlength]{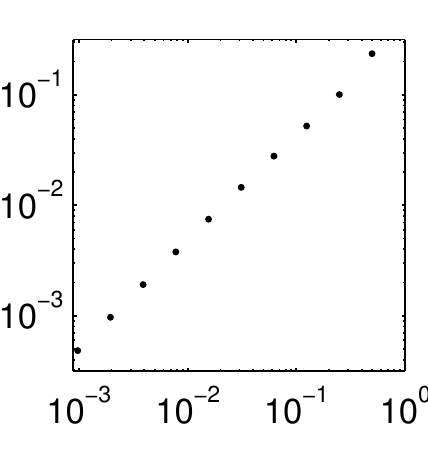}}
\put(1.00,-0.1){\includegraphics[trim=190 10  310 148,clip,width=1\unitlength]{convergenceTest}}
\put(2.00,-0.1){\includegraphics[trim=317 10  178 148,clip,width=1.03\unitlength]{convergenceTest}}
\put(3.00,-0.1){\includegraphics[trim=444 10  51  148,clip,width=1.03\unitlength]{convergenceTest}}
\put(-.02,1.00){\includegraphics[trim=60  156 440 21 ,clip,width=1\unitlength]{convergenceTest}}
\put(1.00,1.00){\includegraphics[trim=190 156 310 21 ,clip,width=1\unitlength]{convergenceTest}}
\put(2.00,1.00){\includegraphics[trim=317 156 183 21 ,clip,width=1\unitlength]{convergenceTest}}
\put(3.03,1.00){\includegraphics[trim=448 156 52  21 ,clip,width=1\unitlength]{convergenceTest}}
\put(0.68,.13){\scalebox{.6}{\footnotesize $\|\x-\x_\tau\|_2$}}
\put(1.21,.13){\scalebox{.6}{\footnotesize $|K\!(\frac1{\!K\!}\mathrm{LOG})_{\!\x_A\!}(\x_B)\!-\!\log_{\!\x_A\!}(\x_B)|$}}
\put(2.28,.13){\scalebox{.6}{\footnotesize $|\Exp{K\!}_{\x_A\!}(\!\tfrac{v}K\!)-\exp_{\x_A\!}(v)|$}}
\put(3.29,.13){\scalebox{.6}{\footnotesize $|K\ParTp_{\x_\tau}\frac{w}K-\parTp_{(\x(t))_{t\in [0,1]}} w|$}}
\end{picture}
\caption{Top: Discrete geodesic path $\x_\tau = (\x_0,\ldots, \x_K)$ with $\x_0=\x_A=(\frac12,0)$ and $\x_1=\x_B=(-\frac12,2)$ on the stereographic projection of the sphere,
$K\Log{K}_{\x_A}(\x_B)$, $\Exp{K}_{\x_A}(v/K)$ for $v=\log_{\x_A}(\x_B)$,
and parallel transport of $w=(-\frac25,0)$ from $\x_A$ to $\x_B$
at different temporal resolutions ($K=2^k$ for $k=1,\ldots,10$).
Bottom: The corresponding errors versus $1/K$.}
\label{fig:convergence}
\end{figure}
In this section we study some properties of the discrete and the continuous path energy on a specific class of  manifolds, and we examine their relation via the concept of $\Gamma$-convergence. In particular we will show that sequences of successively refined discrete geodesic paths converge to a continuous geodesic path.

We consider the following functional analytic set up.
Let $\X$ be a Banach space with norm $\|\cdot\|_\X$
on which a symmetric bilinear form $(\x,v,w) \mapsto g_{\x}(v,w)$ for
$\x,\,v,\,w \in\X$ induces a Riemannian structure with
$g: \X \times \X \times \X \to \R \cup \{\infty\}$.
Furthermore, let $\V$ be a compactly embedded, separable, reflexive
subspace of $\X$ with norm $\|\cdot\|_\V$.
We shall suppose that $g$ is uniformly bounded and $\V$-coercive in the sense
\beqn\label{eq:boundg}
c^\ast \|v\|^2_\V \leq g_\x(v,v) \leq C^\ast \|v\|^2_\V
\eeqn
for constants $c^\ast>0$ and $C^\ast<\infty$ independent of $\x$.
In particular, $g_\x(v,v)=\infty$ if $v\not\in\V$.
Also, $g$ shall be continuous in $\x$ with
\beqn\label{eq:beta}
| g_\x(v,v) - g_{\tilde \x}(v,v)| \leq \beta(\| \x - \tilde \x\|_\X) \|v\|^2_\V
\eeqn
for a strictly increasing, continuous function $\beta$ with $\beta(0)=0$.
Finally, given $\x_A\in\X$ we consider the manifold $\manifold := \x_A + \V$,
on which $g_{\cdot}(\cdot,\cdot)$ induces a Riemannian structure.

\begin{remark}
To motivate this setup, we refer to the infinite-dimensional shape space of rods discussed in detail later in Section \ref{sec:discreteshells}
(the application to finite-dimensional manifolds is explained in the next remark).
Rods $\x$ are considered as differentiable curves in the space $\X=C^1$, and the 
associated metric $g_\x(v,v)$ depends on first and second derivatives of the motion field $v$ and turns out to be finite on $\V = W^{2,2}$.
The rod manifold $\manifold$ is thus locally spanned by all those curves which can be reached from a particular
$\x_A \in \X$ via (short time) integration of $\dot \x(t) = v(t)$ with $v(t) \in W^{2,2}$,
\ie the space $\V$ encodes the regularity of admissible variations.
\end{remark}

\begin{remark}
The assumptions particularly cover the local theory of smooth $k$-dimen\-sional manifolds ($k<\infty$), which by Nash's theorem can
be isometrically embedded in the Euclidian space $\R^m$ for $m$ sufficiently large.
In this case one chooses $\X=\V=\R^k$ to represent a chart of the manifold and
$g_{\x}(v,v) = DX(\x) v \cdot DX(\x) v$, where $X$ is the associated parametrization and
$\cdot$ is the scalar product in $\R^m$.
\end{remark}

For $\x_A,\,\x_B\in\manifold$, the next theorem states the existence of a connecting path with least energy.
The key point is the weak lower semi-continuity of the continuous path energy \eqref{eq:contenergy} using the compact embedding of $\V$ into $\X$.

\begin{theorem}[Existence of continuous geodesics]\label{thm:existenceContinuous}
Under the above assumptions there exists a continuous geodesic path $(\x(t))_{t\in[0,1]}$ defined as a minimizer of the path energy $\energy$ among all paths $(\tilde \x(t))_{t\in[0,1]}$ with
$\tilde \x(0)= \x_A$ and $\tilde \x(1)=\x_B$.
\end{theorem}

\begin{proof}
Let $(\x^j(t))_{t\in[0,1]}$ be a minimizing sequence of paths.
We can obviously assume that the energy on this sequence is bounded by
$\bar\energy = \int_0^1g_{\x_A+t(\x_B-\x_A)}(\x_B-\x_A,\x_B-\x_A)\,\d t$. Hence,
$\|v^j\|_{L^2((0,1);\V)}$ is bounded for $v^j=\dot\x^j$.
Because $L^2((0,1);\V)$ is separable and reflexive, there is a subsequence---with a slight misuse of notation indexed as before---which converges weakly to some
$v \in L^2((0,1);\V)$.
Likewise, the curves $\x^j(t) = \int_0^t v^j(s) \d s$
converge weakly in $W^{1,2}((0,1);\x_A+\V)$ to $\x(t) = \int_0^t v(s) \d s$.
Due to the bounded embedding of $W^{1,2}((0,1);\V)$ into $C^{0,\frac12}([0,1];\V)$
and the compact embedding of $C^{0,\frac12}([0,1];\V)$ into $C^0([0,1];\X)$ by Arzel\`a--Ascoli,
upon selecting a further subsequence, $\x^j$ converges strongly in $C^0([0,1];\X)$.
Now choose $j$ large enough that $\energy[\x^j] \leq \underline \energy + \delta$ and
$\sup_{t\in [0,1]}\beta(\|\x^j(t)-\x(t)\|_\X)\leq \delta$ for a prescribed small $\delta>0$ and
$\underline \energy = \inf \energy$. Following the usual paradigm of the direct methods in the calculus of variations we can estimate the energy using Mazur's lemma, the convexity of $v\mapsto g_\x(v,v)$, the continuity and coercivity of $g$, and Fatou's lemma,
{\allowdisplaybreaks
\beqa
\energy[(\x(t))_{t\in[0,1]}] &=& \int_0^1 g_{\x(t)}(v(t),v(t)) \d t \\
&\leq& \int_0^1 \lim_{j\to\infty} \sum_{i=j}^{N_j}  \lambda_i^j g_{\x(t)}(v^i(t),v^i(t)) \d t \\
&\leq& \int_0^1 \lim_{j\to\infty} \sum_{i=j}^{N_j}  \lambda_i^j g_{\x^i(t)}(v^i(t),v^i(t))\!+\!
\lambda_i^j \beta(\|\x^i(t)\!-\!\x(t)\|_\X) \|v^i(t)\|_\V^2\d t \\
&\leq& \lim_{j\to\infty} \sum_{i=j}^{N_j}\lambda_i^j
\left( \energy[(\x^i(t))_{t\in[0,1]}] + \delta \int_0^1 \|v^i(t)\|^2_\V\,\d t \right) 
\leq \underline \energy \left(1+ \frac{\delta}{c^\ast} \right)\,.
\eeqa
}
Here, $(N_j)_{j=1,\ldots,\infty}$ is a sequence in $\N$ and 
$(\lambda_i^j)_{i=j,\ldots,N_j}^{j=1,\ldots, \infty}$ is a sequence of convex combination coefficients 
with $\sum_{i=j}^{N_j} \lambda_i^j = 1$, $\lambda_i^j\geq 0$, and
$\sum_{i=j}^{N_j} \lambda_i^j v^i\to v$ strongly in  $L^2((0,1);\V)$ for $j\to\infty$.
Because $\delta$ is arbitrary, we obtain $\energy[(\x(t))_{t\in[0,1]}] \leq \underline \energy$, which proves the claim.
\end{proof}

For $\x_A,\,\x_B\in\X$, let us introduce the Riemannian distance
\beqn\label{eqn:RiemannianDistance}
\dist(\x_A,\x_B)=\sqrt{\min_{\x(0)=\x_A,\x(1)=\x_B}\energy[(\x(t))_{t\in[0,1]}]}\,.
\eeqn
It is an easy exercise to verify the axioms of a metric 
and that the induced topology is equivalent to the $\V$-topology,
$\sqrt{c^\ast}\|\x_B-\x_A\|_\V\leq\dist(\x_A,\x_B)\leq\sqrt{C^\ast}\|\x_B-\x_A\|_\V$.
Furthermore, a simple reparameterization argument shows $g_{\x(t)}(\dot\x(t),\dot\x(t))=\mathrm{const.}$ along a geodesic.

For the discrete path energy, we would like to show an analogous existence result,
as well as properties related to \eqref{eqn:RiemannianDistance} and the above-mentioned constant speed parameterization.
For this purpose we consider a lower semi-continuous $\W:\X\times \X\to\R\cup\{\infty\}$ satisfying
\beqn
\W[\x,\tilde\x]=\dist^2(\x,\tilde\x)+O(\dist^3(\x,\tilde\x))
\eeqn
with uniform constants, and we assume coercivity of $\W$ in the sense
\beqn\label{eqn:coercivity}
\W[\x,\tilde\x] \geq \gamma(\dist(\x,\tilde\x))
\eeqn
for a strictly increasing, continuous function $\gamma$ with $\gamma(0)=0$.

\begin{theorem}[Existence of discrete geodesics]\label{thm:existence}
Given $\x_A, \, \x_B \in\manifold$, there is a discrete geodesic path
$(\x_0,\ldots, \x_K)$ which minimizes the discrete energy $\E$ over all
discrete paths $(\tilde \x_0,\ldots, \tilde \x_K)$ with
$\tilde \x_0 = \x_A$ and $\tilde \x_K = \x_B$.
\end{theorem}

\begin{proof}
Let $((\x^j_0,\ldots, \x^j_K))_{j=1,\ldots,\infty}$ be a minimizing sequence.
We can obviously assume the $\x_k^j$ to lie on $\manifold=\x_A+\V$ and that the energy on this sequence is bounded by
$\bar \E = K\W[\x_A, \x_B]+K(K-1)\W[\x_B,\x_B]$. Because $\V$ compactly embeds in $\X$, there is a subsequence, still denoted $(\x^j_0,\ldots, \x^j_K)$, which converges in $\X$ to a discrete path
$(\x_0,\ldots, \x_K)$. By the lower semi-continuity of $\W$ in both arguments, we finally
obtain
$\E[(\x_0,\ldots, \x_K)] \leq \inf_{(\tilde \x_0,\ldots, \tilde \x_K)}
\E[(\tilde \x_0,\ldots, \tilde \x_K)]$.
\end{proof}

\begin{theorem}[Bounds on discrete path energy]\label{thm:energyConvergenceRate}
The discrete path energy satisfies
$$\min_{\substack{(\x_0,\ldots,\x_K)\\\x_0=\x_A,\x_K=\x_B}}\E[(\x_0,\ldots,\x_K)]=\dist^2(\x_A,\x_B)(1+O(\dist(\x_A,\x_B)/K))$$
with uniform constants.
\end{theorem}

\begin{proof}
Denote the minimum energy value by $\underline\E$.
Choosing points $\x_1,\ldots,\x_{K-1}$ such that $\dist(\x_{k-1},\x_k)=\dist(\x_0,\x_K)/K$ for $k=1,\ldots,K$,
we directly see $$\textstyle\underline\E\leq K\sum_{k=1}^K\W[\x_{k-1},\x_k]=\dist^2(\x_0,\x_K)(1+O(\dist(\x_0,\x_K)/K))\,.$$

On the other hand, letting $a_k=\dist(\x_{k-1},\x_k)$ for the minimizing $(\x_0,\ldots,\x_K)$,
we have $\underline\E\geq K\sum_{k=1}^Ka_k^2-Ca_k^3=:F(a_1,\ldots,a_K)$ for some $C>0$.
From the coercivity of $\W$ we know $\underline\E\geq K\sum_{k=1}^K\gamma(a_k)$ and thus $a_k\leq\gamma^{-1}(\dist^2(\x_0,\x_K)(\frac1K+O(\frac{\dist(\x_0,\x_K)}{K^2})))$.
Thus, for $K$ large enough, $a_k\leq1/3C$ for all $k$.
Minimization of $F$ under the constraints $\sum_{k=1}^Ka_k\geq\dist(\x_0,\x_K)$ and $a_k\leq1/3C$, $k=1,\ldots,K$,
yields $a_1=\ldots=a_K=\dist(\x_0,\x_K)/K$
and thus $\underline\E\geq F(a_1,\ldots,a_K)=\dist^2(\x_0,\x_K)(1-C\dist(\x_0,\x_K)/K)$.
\end{proof}

\begin{theorem}[Equidistribution of points along discrete geodesics]\label{thm:equidistribution}
Discrete geodesics (minimizers from Theorem \ref{thm:existence}) satisfy
$\dist(\x_{k-1},\x_k)\leq C\dist(\x_A,\x_B)/K$ for $k=1,\ldots,K\,.$%
\end{theorem}
\begin{proof}
For given $K$, let $j_K\in\{1,\ldots,K\}$ be
such that $\gamma_K:=\dist(\x_{j_K-1},\x_{j_K})$ is largest.
Furthermore abbreviate $d:=\dist(\x_0,\x_K)$ and $\alpha_K:=\dist(\x_0,\x_{j_K-1})/d$.
We have
\beqa
\hspace*{2ex}
&&\hspace*{-2ex}d^2(1\!+\!O(d/K))
=\E[(\x_0,\ldots,\x_K)] \\
&&\geq\tfrac{K}{j_K-1}\E[(\x_0,\ldots,\x_{j_K-1})]+K\gamma_K^2(1-O(\gamma_K))+\tfrac{K}{K-j_K}\E[(\x_{j_K},\ldots,\x_K)]\\
&&\geq K\gamma_K^2(1\!-\!O(\gamma_K))\!+\!\tfrac{K}{j_K\!-\!1}\dist^2(\x_0,\x_{j_K\!-\!1})(1\!-\!O(\gamma_K))
\!+\!\tfrac{K}{K\!-\!j_K}\dist^2(\x_{j_K},\x_K)(1\!-\!O(\gamma_K))\\
&&\geq (1-O(\gamma_K))(K\gamma_K^2+d^2[\tfrac{K}{j_K-1}\alpha_K^2+\tfrac{K}{K-j_K}(1-\alpha_K-\tfrac{\gamma_K}{d})^2])\,.
\eeqa
This is minimized by $\alpha_K=\frac{(j_K-1)(1-\frac{\gamma_K}{d})}{K-1}$
and yields
\beqa
d^2(1+O(d/K))
\geq (1-O(\gamma_K))(K\gamma_K^2+d^2\tfrac{K}{K-1}(1-\tfrac{\gamma_K}{d})^2)\,.
\eeqa
For $\gamma_K=\tfrac{\omega_K}{K}$ with $\omega_K \to \infty$ for $K\to \infty$
and $\omega_K \leq C \sqrt{K}$ by the trivial estimate $\gamma_K \leq \tfrac{C}{\sqrt{K}}$ 
one obtains 
$d^2 + \tfrac{Cd^3}{K} \geq  (1-\tfrac{C}{\sqrt{K}}) (\tfrac{\omega_K^2}{K} +d^2 \tfrac{K}{K-1}
(1-\tfrac{\omega_K}{dK})^2)$
which yields a contradiction, so $\dist(\x_{j_K-1},\x_{j_K})\leq Cd/K$.
\end{proof}

For the subsequent estimates it is convenient to relate the function $\W$ to the metric $g$; we use the following.

\begin{lemma}[Consistency conditions]\label{thm:consistent}
If $(w,v)\mapsto\W[\x\!+\!w,\x\!+\!v]$ is twice G\^ateaux-dif\-ferentiable on $\V\!\times\!\V$ for $\x\!\in\!\X$,
then $\W[\x,\tilde\x]=\dist^2(\x,\tilde\x)+O(\dist^3(\x,\tilde\x))$ for $\tilde\x$ close to $\x$
implies\vspace{-2ex} $$\W[\x,\x] = 0\,,\quad \W_{,2}[\x,\x](v) = 0\,,\quad \W_{,22}[\x,\x](v,w) = 2g_\x(v,w)$$ for any $v,w\in\V$.
Furthermore,
$\W_{,1}[\x,\x](v) = 0$ and  $$\W_{,11}[\x,\x](v,w) = -\W_{,12}[\x,\x](v,w) = -\W_{,21}[\x,\x](v,w) = \W_{,22}[\x,\x](v,w)\,.$$
If $(w,v)\mapsto\W[\x+w,\x+v]$ is even three times Fr\'echet-differentiable, the implication becomes an equivalence.
\end{lemma}

\begin{proof}
It is readily shown that $\dist^2(\x,\tilde\x)=g_\x(\tilde\x-\x,\tilde\x-\x)+O(\beta(\|\tilde\x-\x\|_\X))\|\tilde\x-\x\|_\V^2$.

Let $\W[\x,\tilde\x]=\dist^2(\x,\tilde\x)+O(\dist^3(\x,\tilde\x))$ for $\tilde\x$ close to $\x$,
then obviously $\W[\x,\x]=0$ and $\W_{,2}[\x,\x](v)=\frac\d{\d t}\W[\x,\x+tv]=0$.
Now assume $\W_{,22}[\x,\x](v,v)\neq 2g_\x(v,v)$ for some $v\in\V$
(note that due to the bilinearity it is sufficient to show equality of the quadratic forms).
Without loss of generality, let $\W_{,22}[\x,\x](v,v)>\alpha g_\x(v,v)$ for some $\alpha>2$.
This implies $\W[\x,\x+tv]> \frac12\left(\frac{\alpha}{2} +1\right) g_\x(tv,tv)$
for $t$ small enough,
which (using $\dist^2(\x,\x+tv)=g_\x(tv,tv)+o(t^2)$)
is strictly greater than $(\tfrac{\alpha}{4}+\tfrac12) \dist^2(\x,\x+tv)+o(\dist^2(\x,\x+tv))$ and thus a contradiction.

The above applied to the first argument of $\W$ instead of the second implies $\W_{,1}[\x,\x]=0$ and $\W_{,11}[\x,\x](v,w)=2g_\x(v,w)$.
Finally, for any curve $(\x(t))_{t\in\R}$ in $\X$ we can differentiate $0=\W_{,1}[\x(t),\x(t)]$ and $0=\W_{,2}[\x(t),\x(t)]$ with respect to $t$,
yielding $\W_{,12}[\x,\x]=-\W_{,11}[\x,\x]$ and $\W_{,21}[\x,\x]=-\W_{,22}[\x,\x]$.

If $v\mapsto\W[\x,\x+v]$ is three times Fr\'echet-differentiable and $\W_{,22}[\x,\x](v,w) = 2g_\x(v,w)$,
then by Taylor's theorem for $\tilde\x=\x+tv$, $\W[\x,\tilde\x]=g_\x(tv,tv)+O(t^3)=\dist^2(\x,\tilde\x)+O(\dist^3(\x,\tilde\x))$.
\end{proof}

\begin{theorem}[Uniqueness of discrete geodesics]\label{thm:uniqueness}
If $(w,v)\mapsto\W[\x+w,\x+v]$ is twice Fr\'echet-differentiable on $\V\times \V$ for $\x\in\manifold$
and if $\|\x_B-\x_A\|_{\V}$ is sufficiently small 
there exists a unique discrete geodesic $(\x_0,\ldots, \x_K)$  with $\x_0=\x_A$ and $\x_K=\x_B$.
\end{theorem}

\begin{proof} Without any restriction we assume $\x_A \in \V$. Otherwise, we have to work with offsets $\x_k-\x_A$ 
instead of points $\x_k$.
For $X_{0,K} := (\x_0,\x_K)$ and  $X_{1,\ldots,K-1} := (\x_1,\ldots,\x_{K-1})$ we define
the function
$F: \manifold^{K+1} \to (\V')^{K-1}; (X_{0,K},  X_{1,\ldots,K-1}) \mapsto
(\W_{,2}[\x_{k-1},\x_k] + \W_{,1}[\x_{k},\x_{k+1}])_{k=1,\ldots,K-1}\,,$
where $\V'$ denotes the dual space of $\V$.
Then the block tridiagonal operator $\mathcal{A}= D_{X_{1,\ldots,K-1}} F(X_{0,K},  X_{1,\ldots,K-1})$ is a $(K\!-\!1)\times (K\!-\!1)$
block operator given by
\beqa
\mathcal{A}_{kk} &=& \W_{,22}[\x_{k-1},\x_k] + \W_{,11}[\x_{k},\x_{k+1}]\,, \\
\mathcal{A}_{k(k-1)} &=& \W_{,21}[\x_{k-1},\x_k]\,, \quad
\mathcal{A}_{k(k+1)} = \W_{,12}[\x_{k},\x_{k+1}] \,.
\eeqa
Now, using Lemma \ref{thm:consistent}
we obtain for the trivial geodesic $(\x_A,\ldots, \x_A)$  that
\beq
\mathcal{A} = 2 \left( \begin{smallmatrix}
2 g & -g & & & \\
-g & 2 g &-g & & \\
& \ddots & \ddots &\ddots & & \\
&& -g & 2 g &-g & \\
& & &  -g & 2 g \\
\end{smallmatrix}\right)
\!= 2\,  \mathrm{diag}(g^{\frac12})  \! \left( \begin{smallmatrix}
2 \Id & -\Id & & & \\
-\Id & 2 \Id &-\Id & & \\
& \ddots & \ddots &\ddots & & \\
&& -\Id & 2 \Id &-\Id & \\
& & &  -\Id & 2 \Id \\
\end{smallmatrix}\right)
\! \mathrm{diag}(g^{\frac12})\,,
\eeq
where $g=g_{\x_A}$. Hence, the inverse of $\mathcal{A}$ can be computed by
Gaussian elimination, which ensures that $\mathcal{A} = D_{X_{1,\ldots,K-1}} F(\x_A,\ldots, \x_A)$ is invertible with bounded inverse, and
thus the claim follows by the implicit function theorem.
\end{proof}

In what follows, we aim to prove convergence of discrete geodesics against continuous ones.
To this end we identify any discrete path $(\x_0,\ldots,\x_K)$ on $\manifold$
with its piecewise geodesic interpolation $\hat \x^K:[0,1]\to\manifold$, \ie\ every segment $\hat \x^K|_{[k\tau,(k+1)\tau]}$ shall be the shortest continuous connecting geodesic between $\x_k$ and $\x_{k+1}$.
Now, we define an energy $\E^K$ on continuous paths via
$\E^K[(\x(t))_{t\in[0,1]}]:=\E[(\x_0,\ldots,\x_K)]$
if $\x(t)=\hat \x^K(t)$ for some $\x_0,\ldots,\x_K\in\manifold$ with $\x_0= \x_A$, $\x_K= \x_B$,
and $\E^K[(\x(t))_{t\in[0,1]}]=\infty$ else.
Based on these notational preliminaries we obtain the following convergence result.

\begin{theorem}[$\Gamma$-convergence of the discrete energy]\label{thm:Gamma}
In the $L^2((0,1);\X)$-topology the $\Gamma$-limit of $\E^K$ for $K\to\infty$ is $\energy\,.$
\end{theorem}

\begin{proof}
To verify $\Gamma$-convergence we have to establish the two defining properties,
the limsup- and the liminf-inequality.

\noindent [\emph{limsup-inequality}] For an arbitrary $\x:[0,1]\to\manifold$ with $\x \in L^2((0,1);\X)$
we have to show that there exists a sequence of curves $\hat \x^K:[0,1]\to\manifold$ with
$\hat\x^K\to\x$ in $L^2((0,1);\X)$ and
$\limsup_{K\to \infty} \E^K[(\hat \x^K(t))_{t\in[0,1]}] \leq \energy[(\x(t))_{t\in[0,1]}]\,.$
Without any restriction we can assume $\energy[(\x(t))_{t\in[0,1]}]<\infty$.
Applying the Cauchy--Schwarz inequality we get
\beqan
\dist(\x(t),\x(s)) &=& \int_t^s \!\!\! \sqrt{g_{\x(r)}(\dot\x(r),\dot\x(r))} \d r 
\textstyle\leq \sqrt{|s-t| \int_t^s g_{\x(r)}(\dot\x(r),\dot\x(r)) \d r } \,,\; \label{eq:CS}
\eeqan
which immediately implies $W^{1,2}((0,1);\manifold)\subset C^{0,\frac12}([0,1];\manifold)$
and thus H\"older continuity of paths with finite path energy.
Now, let $\hat \x^K$ denote the piecewise geodesic interpolation of
$(\x(\frac{0}K),\ldots,\x(\frac{K}K))$.
For any $K$ we have
\beqa
\energy[(\x(t))_{t\in[0,1]}] &\geq&\textstyle K\sum_{k=1}^K\dist^2(\x(\tfrac{k-1}K),\x(\tfrac{k}K)) \\
&\geq&\textstyle K\sum_{k=1}^K \W[\x(\tfrac{k-1}K),\x(\tfrac{k}K)] - C \,K\, \sum_{k=1}^K \dist^3(\x(\tfrac{k-1}K),\x(\tfrac{k}K))) \\
&\geq&\textstyle \E^K[(\hat \x^K(t))_{t\in[0,1]}] - C \,K\, \sum_{k=1}^K K^{-\frac32} \energy[(\hat \x^K(t))_{t\in[\frac{k-1}{K},\frac{k}{K}]}]^{\frac32} \\
&\geq&\textstyle \E^K[(\hat \x^K(t))_{t\in[0,1]}] \big(1 - C \,K^{-\frac12} \sqrt{\energy[(\x(t))_{t\in[0,1]}]}\big)\,,
\eeqa
where we have used \eqref{eq:CS}. Letting $K\to\infty$ yields the desired limsup--inequality.

\noindent [\emph{liminf-inequality}] We have to show that for any sequence of curves $(\x^K)_{K\geq1}$ with
$\x^K:[0,1]\to\manifold$ and $\x^K \rightarrow \x$ in $L^2((0,1);\X)$ the inequality
$
\liminf_{K\to \infty} \E^K[(\x^K(t))_{t\in[0,1]}] \geq \energy[(\x(t))_{t\in[0,1]}]
$
holds. Without any restriction we assume that
$
\E^K[(\x^K(t))_{t\in[0,1]}]  
< \bar \E < \infty
$
uniformly.
Thus, $\W[\x^K(\frac{k-1}K),\x^K(\frac{k}K)]\rightarrow 0$ uniformly as $K\to\infty$ so
that due to the coercivity  $\dist(\x^K(\frac{k-1}K),\x^K(\frac{k}K))\rightarrow 0$ uniformly as well.
Next, we estimate
\beqa
&& \hspace{-4ex}\liminf_{K\to\infty}\E^K[(\x^K(t))_{t\in[0,1]}] =
\liminf_{K\to\infty} K {\textstyle\sum_{k=1}^K}   \W[\x^K(\tfrac{k-1}K),\x^K(\tfrac{k}K)] \\
&&\geq \liminf_{K\to\infty} K {\textstyle\sum_{k=1}^K} \dist^2(\x^K(\tfrac{k-1}K),\x^K(\tfrac{k}K))
\left(1- C \, \dist(\x^K(\tfrac{k-1}K),\x^K(\tfrac{k}K)) \right) \\
&&=\liminf_{K\to\infty}\energy[(\x^K(t))_{t\in[0,1]}]\,,
\eeqa
which also shows that $\x^K$ is uniformly bounded in $W^{1,2}((0,1);\manifold)$.
Due to the reflexivity of $\V$, a subsequence (for simplicity again denoted by $(\x^K)_{K\geq 1}$) weakly converges against $\x$
in $W^{1,2}((0,1);\manifold)$.
Due to the sequential weak lower semi-continuity of the energy $\energy$ (\cf Theorem\,\ref{thm:existenceContinuous}) we finally
obtain the requested estimate
\beqa
\liminf_{K\to\infty}\E^K[(\x^K(t))_{t\in[0,1]}] \geq
\liminf_{K\to\infty}\energy[(\x^K(t))_{t\in[0,1]}] \geq
\energy[(\x(t))_{t\in[0,1]}] \,,
\eeqa
which concludes the proof.
\end{proof}

\begin{corollary}[Convergence of discrete geodesics]\label{thm:Linfty}
Minimizers of the discrete path energies $\E^K$ converge against minimizers of the continuous path energy $\energy$
in $C^0([0,1];\X)$.
\end{corollary}

\begin{proof}
This is a simple implication of the $\Gamma$-convergence and the following equi-mild coercivity of the discrete energies:
Theorem\,\ref{thm:energyConvergenceRate} shows that the minima of the discrete energies are uniformly bounded.
However, $\E^K$ is coercive with respect to $\x_A+W^{1,2}((0,1);\V)$ so that the minima for all $K$ are achieved in some bounded ball of $\x_A+W^{1,2}((0,1);\V)$,
which is compact in $L^2((0,1);\X)$.
Together with the $\Gamma$-convergence of $\E^K$,
this implies that any $L^2((0,1);\X)$-limit point of a sequence of minimizers $\hat\x^K$ for $\E^K$ is a minimizer for $\energy$.
Furthermore, from the above, all converging subsequences are bounded in $\x_A+W^{1,2}((0,1);\V)$
and thus bounded in $\x_A+C^{0,\frac12}([0,1];\V)$ and by Arzel\`a--Ascoli precompact in $C^0([0,1];\X)$
so that the convergence is not only in $L^2((0,1);\X)$, but even in $C^0([0,1];\X)$.
\end{proof}

The above convergence obviously does not only hold
for the piecewise geodesic interpolation $\hat\x^K$ of discrete geodesics $(\x_0,\ldots,\x_K)$,
but also for the piecewise linear interpolation which we shall call $\x_\tau$ (where $\tau=1/K$ stands for the discrete time step).

For stronger convergence estimates and for the convergence of discrete logarithm, exponential map, and parallel transport,
we require the following smoothness hypotheses:

\begin{enumerate}
\renewcommand{\theenumi}{(H\arabic{enumi})}
\renewcommand*{\labelenumi}{\theenumi}
\item The metric $g$ is $C^2(\X;\V'\otimes \V')$-smooth.\label{enm:smoothMetric}
\item The energy $\W$ is $C^4((\x_A+\V)\times(\x_A+\V);\R)$-smooth with bounded derivatives.\label{enm:smoothEnergy}
\end{enumerate}

The following theorem now states that the convergence in $C^0([0,1];\X)$ ensured by the above $\Gamma$-convergence result is actually much stronger with $L^2$-converging velocities.

\begin{theorem}[Path convergence in $W^{1,2}((0,1);\manifold)$]\label{thm:W12}
Under the hypotheses \ref{enm:smoothMetric} and \ref{enm:smoothEnergy},
if the interpolated geodesics $\x_\tau(\cdot)$ converge for $K\to \infty$ to the continuous one $\x(\cdot)$ in $L^2((0,1);\X)$,
then this convergence is even in $W^{1,2}((0,1);\manifold)$.
\end{theorem}

\begin{proof}
Under the smoothness and coercivity assumptions on $g$, a continuous
geodesic path $(\x(t))_{t \in [0,1]}$ in $\manifold$
as a minimizer of the continuous energy $\int_0^1 g_{\x(t)}(\dot \x(t),\dot \x(t)) \d t$ for given fixed end points $\x(0)$ and $\x(1)$
is smooth and fulfills the Euler--Lagrange equation
\beqn
\int_0^1 2 g_{\x(t)}(\dot \x(t),\dot \psi(t)) +
\left(D_\x g_{\x(t)}(\psi(t))\right) (\dot \x(t),\dot \x(t)) \d t = 0
\label{eq:contEL}
\eeqn
for all $\psi \in W^{1,2}_0((0,1);\V) := \setof{\psi \in W^{1,2}((0,1);\V)}{\psi(0) = \psi(1) = 0}$.
Furthermore, the Euler--Lagrange equation for a discrete geodesic path is given by
\beqn\textstyle
K \sum_{k=1}^{K} \left( \W_{,1}[\x_{k-1},\x_k] (\psi_{k-1}) + \W_{,2}[\x_{k-1},\x_k] (\psi_{k}) \right)
=0
\label{eq:discreteEL}
\eeqn
for all $(\psi_{k})_{k=0,\ldots, K} \subset \V$ with $\psi(0) = \psi(K) = 0$.
Applying the Taylor expansion
\begin{multline}
\W[\x_{k-1},\x_k] = \W[\x_{k-1},\x_{k-1}] +
\W_{,2}[\x_{k-1},\x_{k-1}] (\x_k-\x_{k-1}) \\
+ \int_0^1 (1-s) \W_{,22}[\x_{k-1},\x_{k-1}+ s(\x_k-\x_{k-1})]
(\x_k-\x_{k-1},\x_k-\x_{k-1}) \d s\,,\label{eqn:TaylorExpansion}
\end{multline}
whose first two terms on the right-hand side vanish,
one can rewrite \eqref{eq:discreteEL} as
\beqa
0&=& \frac1{\tau} \sum_{k=1}^{K}  \int_0^1 (1-s)\Big(
W_{,221}[\x_{k-1},\x_{k-1} \!+\! s (\x_k\!-\!\x_{k-1})]((\x_k\!-\!\x_{k-1}),(\x_k\!-\!\x_{k-1}),\psi_{k-1}) \\[-1ex]
&&\qquad  \qquad \quad \;\;  + W_{,222}[\x_{k-1},\x_{k-1} \!+\! s (\x_k\!-\!\x_{k-1})]\\
&& \qquad  \qquad \qquad \qquad \quad
((\x_k\!-\!\x_{k-1}),(\x_k\!-\!\x_{k-1}),\psi_{k-1} + s(\psi_{k}\!-\!\psi_{k-1})) \\
&&\qquad  \qquad \quad \;\;  + 2\, W_{,22}[\x_{k-1},\x_{k-1} \!+\! s (\x_k\!-\!\x_{k-1})]((\x_k\!-\!\x_{k-1}),(\psi_k\!-\!\psi_{k-1}))
\Big)\, \d s \\
&=& \sum_{k=1}^{K}  \int_0^\tau (1-t/\tau)\Big(
W_{,221}[\x_{k-1},\x_{k-1} \!+\! t v_k](v_k,v_k,\psi_{k-1}) \\[-1ex]
&&\qquad \qquad \quad + W_{,222}[\x_{k-1},\x_{k-1} \!+\! t v_k](v_k,v_k,\psi_{k-1} + t w_k) \\
&&\qquad \qquad \quad + 2\, W_{,22}[\x_{k-1},\x_{k-1} \!+\! t v_k](v_k,w_k)
\Big)\, \d t
\eeqa
for $v_k = \frac{\x_k\!-\!\x_{k-1}}{\tau}$ and $w_k = \frac{\psi_k\!-\!\psi_{k-1}}{\tau}$.
Now, taking into account the smoothness of the geodesic, the smoothness of
$\W$, and $g_\x(v,w) = \frac12 \W_{,22}[\x,\x](v,w)$,
which implies
\begin{equation}\label{eqn:metricDeriv}
2(D_\x g_\x(\psi))(v,w) = \left( \W_{,221}[\x,\x](v,w,\psi) + \W_{,222}[\x,\x](v,w,\psi) \right)\,,
\end{equation}
we finally achieve
\beqan
0&=& \frac12 \sum_{k=1}^{K}  \tau \Big(
W_{,221}[\x_{k-1},\x_{k-1}](v_k,v_k,\psi_{k-1}) + W_{,222}[\x_{k-1},\x_{k-1}](v_k,v_k,\psi_{k-1}) \nonumber\\[-2ex]
&&\qquad \qquad + 2\, W_{,22}[\x_{k-1},\x_{k-1}](v_k,w_k)
\Big)\, + \mathrm{Err}  \nonumber\\[-1ex]
&=& \sum_{k=1}^{K}  \tau \Big( 2 g_{\x_{k-1}}(v_k,w_k) + \left( (D_\x g_{\x_{k-1}})(\psi_{k-1})\Big)(v_k,v_k) \right) + \mathrm{Err}
 \label{eq:discreteEL2}
\eeqan
for $\mathrm{Err} = O(\tau^2\sum_k\|v_k\|_\V^3\|\psi_{k-1}\|_\V) + O(\tau^2\sum_k\|v_k\|_\V^2\|w_k\|_\V)$.
From Theorem\,\ref{thm:equidistribution} we obtain the following estimates for the error,
\begin{align}
\mathrm{Err}&=\textstyle O(\tau^2)\left(\sum_{k=1}^K\|\psi_{k-1}\|_\V+\sum_{k=1}^K\|w_k\|_\V\right)\,,\label{eq:errLog}\\
\mathrm{Err}&=
\textstyle O(\tau)\left(\sqrt{\tau\sum_{k=1}^K \|\psi_{k-1}\|^2_\V}+\sqrt{\tau\sum_{k=1}^K \|w_k\|^2_\V}\right)\,.\label{eq:errH1}
\end{align}
Based on these preliminaries we now combine the Euler--Lagrange equation \eqref{eq:contEL} for the continuous geodesic path and \eqref{eq:discreteEL2} (derived from the Euler--Lagrange equation \eqref{eq:discreteEL} for the discrete geodesic path) to obtain an equation for the discretization error.
To this end, we consider a piecewise polygonal function $\psi_\tau: [0,1] \to\V$ with $\psi_\tau(k\tau) = \psi_k \in\V$ for $k=0,\ldots, K$ and $\psi_0 = \psi_K = 0$.  Using this notation one easily verifies the identity
\beqan
&&\textstyle \int_0^1
2 g_{\x(t)}(\dot \x_\tau(t)\!-\!\dot \x(t),\dot \phi(t)) \d t \nonumber\\
&&\textstyle = \int_0^1
2 g_{\x(t)}(\dot \x_\tau(t)\!-\!\dot \x(t),\dot \phi(t) \!-\! \dot \psi_\tau(t)) \d t \nonumber\\
&&\textstyle\quad  \!-\! \int_0^1
2 g_{\x(t)}(\dot \x(t),\dot \psi_\tau(t)) \!+\!
D_\x g_{\x(t)}(\psi_\tau(t))(\dot \x(t),\dot \x(t)) \d t \nonumber\\
&&\textstyle\quad \!+\! \sum_{k=1}^{K}  \tau \Big( 2 g_{\x_{k\!-\!1}}(v_k,w_k) \!+\! D_\x g_{\x_{k\!-\!1}}(\psi_{k\!-\!1}) (v_k,v_k) \Big) \nonumber\\
&&\textstyle\quad  \!+\! \sum_{k=1}^{K} 2 \left(  \int_{(k\!-\!1)\tau}^{k\tau}
 g_{\x_\tau(t)}(\dot \x_\tau(t),\dot \psi_\tau(t)) \d t \!-\!  \tau g_{\x_{k\!-\!1}}(v_k,w_k) \right) \nonumber\\
&&\textstyle\quad  \!+ 2\int_0^1 g_{\x(t)}(\dot \x_\tau(t),\dot \psi_\tau(t)) \!-\!
g_{\x_\tau(t)}(\dot \x_\tau(t),\dot \psi_\tau(t)) \d t \nonumber\\
&&\textstyle\quad \!+\! \sum_{k=1}^{K}  \left(\int_{(k\!-\!1)\tau}^{k\tau}
D_\x g_{\x_\tau(t)}(\psi_{k\!-\!1})(\dot \x_\tau(t),\dot \x_\tau(t)) \d t \!-\!
 \tau  D_\x g_{\x_{k\!-\!1}}(\psi_{k\!-\!1})(v_k,v_k)\right)\nonumber\\
&&\textstyle\quad \!+\! \sum_{k=1}^{K}  \left(\int_{(k\!-\!1)\tau}^{k\tau} D_\x g_{\x_\tau(t)}(\psi_\tau(t))(v_k,v_k)
- D_\x g_{\x_\tau(t)}(\psi_{k\!-\!1})(\dot \x_\tau(t),\dot \x_\tau(t)) \d t \right)\nonumber\\
&&\textstyle\quad  \!+\! \int_0^1
 \left(D_\x g_{\x(t)}(\psi_\tau(t))\!-\! D_\x g_{\x_\tau(t)}(\psi_\tau(t))\right)(\dot \x_\tau(t),\dot \x_\tau(t)) \d t \nonumber\\
&&\textstyle\quad  \!+\! \int_0^1\!\!
  D_\x g_{\x(t)}(\psi_\tau(t))(\dot \x_\tau(t),\dot \x(t)\!-\!\dot \x_\tau(t)) \!+\!  D_\x g_{\x(t)}(\psi_\tau(t))(\dot \x(t)\!-\!\dot \x_\tau(t),\dot \x(t)) \d t \nonumber \\
&&=:  I - II + III + IV + V + VI + VII + VIII + IX   \label{eq:I-IX}
\eeqan
with $v(t) = \dot \x(t)$, $w(t)=\dot\psi_\tau(t)$.
Now, we choose $\phi=\x_\tau - \x$
and $\psi_\tau = \x_\tau - \mathcal{I}_\tau \x$, where $\mathcal{I}_\tau$ is the piecewise affine Lagrangian interpolation in time with
$\mathcal{I}_\tau \psi(k\tau) = \psi(k\tau)$. Let us mention here, that we will reuse \eqref{eq:I-IX} with different test functions later in the context of a pointwise error estimate in the proof of Theorem\,\ref{thm:Log}.
Due to the uniform coercivity of the metric $g$ the left-hand side can be estimated from below by
$2 c^\ast \|\dot \x - \dot\x_\tau\|_{L^2((0,1);\V)}^2$.
The different terms on the right-hand side of \eqref{eq:I-IX} are estimated as follows:
\begin{itemize}
\setlength{\itemsep}{1ex}
\item[(I)] By the uniform boundedness of the metric, $|\,I\,|$ can be estimated by\\
$2 C^\ast \|\dot \x - \dot \x_\tau\|_{L^2((0,1);\V)}  \|\dot \x - (\Itau\x)\dot\ \|_{L^2((0,1);\V)}$.
\item[(II)] Due to \eqref{eq:contEL} the term $II$ vanishes.
\item[(III)] Taking into account \eqref{eq:discreteEL2} and \eqref{eq:errH1},
$|\,III\,|$ is bounded by $C\tau\|\psi_\tau\|_{W^{1,2}((0,1);\V)}$. 
\item[(IV)] Due to the smoothness of the metric $g$ and Theorem\,\ref{thm:equidistribution}, 
the quadrature error in $|\,IV\,|$ is bounded by 
$O(\tau^2\sum_k\|v_k\|_\V^2\|\dot\psi_\tau\|_\V)=O(\tau\|\dot \psi_\tau\|_{L^2((0,1);\V)})$ (compare \eqref{eq:errH1}).
\item[(V)] By an analogous argument and the boundedness of  $\x_\tau$ in $L^2((0,1);\X)$ the term $|\,V\,|$ can be estimated by
$C \| \x - \x_\tau\|_{L^2((0,1);\X)} \|\dot \x_\tau\|_{L^\infty((0,1);\V)}\|\dot \psi_\tau\|_{L^2((0,1);\V)}$, where we know from Theorem
\ref{thm:equidistribution} that $\|\dot \x_\tau\|_{L^\infty((0,1);\V)}=O(1)$. 
\item[(VI)] Now using the smoothness of $D_\x g$ we can bound $|\,VI\,|$ by \\ 
$C\tau^2\sum_{k=1}^K\|v_k\|_\V^2\|v_k\|_\X\|\psi_{k-1}\|_\X=O(\tau\|\psi_\tau\|_{L^2((0,1);\X)})$.
\item[(VII)] Term $|\,VII\,|$ is bounded by $C \tau \|\dot \psi_\tau\|_{L^2((0,1);\X)}$ due to the uniform bound for $\|\dot\x_\tau\|_\V$.
\item[(VIII)] Now using the smoothness of $D_\x g$ and the boundedness of $\|\dot\x_\tau\|_\V$ we can bound 
$|\,VIII\,|$ by $C\|\x-\x_\tau\|_{L^2((0,1);\X)}\|\psi_\tau\|_{L^2((0,1);\X)}$.
\item[(IX)] Finally, we obtain the estimate
$|\,IX\,| \leq C\|\psi_\tau\|_{L^2((0,1);\X)} \|\dot \x - \dot \x_\tau\|_{L^2((0,1);\V)}$. 
\end{itemize}

\noindent
Altogether, for some constant $C$ we obtain the following estimate for the error,
\beqa
2c^\ast\|\dot \x - \dot \x_\tau\|_{L^2((0,1);\V)}^2
&\leq& 2 C^\ast \|\dot \x - \dot \x_\tau\|_{L^2((0,1);\V)}  \|\dot \x - (\Itau\x)\dot\ \|_{L^2((0,1);\V)}\\
&&+C\tau\|\psi_\tau\|_{W^{1,2}((0,1);\V)}\\
&&+C \| \x - \x_\tau\|_{L^2((0,1);\X)}\|\dot \psi_\tau\|_{L^2((0,1);\V)}\\
&&+C\tau\|\psi_\tau\|_{W^{1,2}((0,1);\X)}\\
&&+C\|\x-\x_\tau\|_{L^2((0,1);\X)}\|\psi_\tau\|_{L^2((0,1);\X)}\\
&&+C \|\psi_\tau\|_{L^2((0,1);\X)} \|\dot \x - \dot \x_\tau\|_{L^2((0,1);\V)}\,.
\eeqa
Finally, using Poincar\'e's inequality $\|\x-\x_\tau\|_{L^2((0,1);\V)}\leq\|\dot\x-\dot\x_\tau\|_{L^2((0,1);\V)}$
as well as Young's inequality and the compact embedding $\V\hookrightarrow \X$, we derive in a straightforward way
\beqa
\|\dot\x-\dot\x_\tau\|_{L^2((0,1);\V)}^2&\leq&C\left(\tau^2+\|\x-\Itau\x\|_{W^{1,2}((0,1);\V)}^2+\|\x-\x_\tau\|_{L^2((0,1);\X)}^2\right)
\eeqa
for some $C>0$.
Taking into account  the classical interpolation estimate
$$\| \x - \Itau \x\|_{L^2((0,1);\V)} + \tau \|\dot \x - (\Itau\x)\dot\ \|_{L^2((0,1);\V)} \leq C \tau^2$$
as well as $\|\x-\x_\tau\|_{L^2((0,1);\X)}\to0$, this concludes the proof.
\end{proof}

\begin{remark}
Note that in the continuous setting, geodesics can also be obtained by minimizing the path length \eqref{eq:contlength}
for fixed end points $\x(0),\x(1)$.
However, discrete minimizers of the discrete path energy \eqref{eq:discretePathEnergy}
are in general unrelated to continuous geodesics.
As an example, think of a smooth curved manifold embedded in $\R^m$
and let $\W[\x,\tilde\x]$ be the squared Euclidean distance between $\x$ and $\tilde\x$ in the embedding space.
Then given $\x_0,\x_K$ with $\x_0+ t (\x_1-\x_0) \not\in \manifold$ for all $t\in (0,1)$, minimizers of $\L[\cdot]$ for $\x_0,\x_K$ fixed
have the form $(\x_0,\x_0,\ldots,\x_0,\x_K,\ldots,\x_K)\,.$
Obviously, the claim of Theorem\,\ref{thm:equidistribution} no longer holds,
and the argument for Theorem\,\ref{thm:Gamma} breaks down
since minimizers of the discrete length $\L[\cdot]$ are only bounded in $BV((0,1);\manifold)$ instead of $W^{1,2}((0,1);\manifold)$.
\end{remark}

\section{Convergence of discrete logarithm, exponential, and parallel transport}
\label{sec:convergence}
In this section we discuss the limit behavior of the discrete operators.
At first we investigate the convergence of the discrete logarithm, which can be formulated as an
$L^\infty$ derivative error estimate for discrete variational solutions of an elliptic problem in
$W^{1,2}((0,1);\manifold)$. This will become apparent in the proof of the following theorem.

\begin{theorem}[Convergence of discrete logarithm]\label{thm:Log}
Let $\x,\tilde\x\in\manifold$.
Under the hypotheses \ref{enm:smoothMetric} and \ref{enm:smoothEnergy},
if the continuous and discrete geodesics between $\x,\tilde\x$ are unique,
$K \Log{K}_\x \tilde \x \rightarrow \log_\x \tilde \x$
weakly in $\V$ (and thus strongly in $\X$) as $K\to \infty$.
\end{theorem}

\begin{proof}
\renewcommand{\v}{{\mathfrak{v}}}
As usual, abbreviate $\tau=1/K$, and let $\x(t)_{t\in[0,1]}$ and $\x_\tau(t)_{t\in[0,1]}$ be the continuous and the interpolated discrete geodesic between $\x$ and $\tilde\x$.
Let us denote by $\Phi\in W^{1,2}_0((0,1);\mathrm{Hom}(\V',\V))$ the weak solution of
\beqn\label{eq:Green}
\int_0^1
2 g_{\x(t)}(\dot z(t),\dot \Phi(t)\v) \d t = \left(\v,\tfrac{z(\tau)}{\tau}\right)_{\V',\V}
\eeqn
for all $z \in W^{1,2}_0((0,1);\V)$, $\v\in\V'$. Then we obtain
\begin{align}
\!\!\!\left(\v,K\Log{K}_{\x}{\tilde \x} - \log_{\x}{\tilde \x}\right)_{\V',\V}
&= \left(\v,\tfrac{\x_\tau(\tau) - \x(0)}{\tau} - \dot \x(0)\right)_{\V',\V} \nonumber\\
&= \left(\v,\tfrac{\x_\tau(\tau)-\x(\tau)}{\tau} + \left(\tfrac{\x(\tau)-\x(0)}{\tau} - \dot \x(0)\right)\right)_{\V',\V} \nonumber\\
&= \!\int_0^1\!\!\!\!2 g_{\x(t)}(\dot \x_\tau(t)\!-\!\dot \x(t),\dot \Phi(t)\v) \d t + O(\tau\|\v\|_{\V'})\,.\label{eq:LogError}
\end{align}
The solution of \eqref{eq:Green} for a given continuous geodesic path $(\x(t))_{t\in [0,1]}$ can be computed explicitly.
Indeed, we deduce from  \eqref{eq:Green} the following two conditions,
\begin{eqnarray}
2 \frac{\d}{\d t} \left(\LMg_{\x(t)}(\dot \Phi(t)\v)\right )&=& 
0 \quad \mbox{on $(0,\tau) \cup (\tau,1)$}\label{eq:Green1}\,,\\
2\, \LMg_{\x(\tau)} \left(\dot \Phi(\tau+0)\v-\dot \Phi(\tau-0)\v\right) &=& 
-\tfrac{1}{\tau} \v \label{eq:Green2}
\end{eqnarray}
for any $\v\in\V'$, where we use the notation $\LMg_{\x}:\V\to\V'$ for
the inverse Riesz isomorphism of the Hilbert space $(\V,g_\x)$.
Integration of \eqref{eq:Green1} leads to $\dot\Phi(t) = \LMg_{\x(t)}^{-1} \LMg_{\x(0)} \dot\Phi(0)$ on $(0,\tau)$
and $\dot\Phi(t) = \LMg_{\x(t)}^{-1} \LMg_{\x(1)} \dot\Phi(1)$ on $(\tau,1)$.
From this and \eqref{eq:Green2} one obtains $\LMg_{\x(1)} \dot\Phi(1)-\LMg_{\x(0)} \dot\Phi(0) = -\frac1{2\tau}$.
Furthermore, due to the boundary conditions $\Phi(0) = \Phi(1) =0$ we achieve
\beqan\label{eq:Green3}
0 &=& \int_0^\tau \!\!\dot \Phi(t) \d t + \!\int_\tau^1 \!\! \dot \Phi(t) \d t 
= \left(\int_0^\tau \!\! \LMg_{\x(t)}^{-1} \d t\right)  \LMg_{\x(0)} \dot\Phi(0)
\!+\! \left(\int_\tau^1 \!\! \LMg_{\x(t)}^{-1} \d t\right)  \LMg_{\x(1)} \dot\Phi(1)\,.
\eeqan
Now, with the notation $G^- = \int_0^\tau \LMg_{\x(t)}^{-1} \d t$, $G^+ = \int_\tau^1 \LMg_{\x(t)}^{-1} \d t$, $A=\LMg_{\x(0)} \dot\Phi(0)$, and $B= \LMg_{\x(1)} \dot\Phi(1)$,
\eqref{eq:Green2} and \eqref{eq:Green3} form the linear system of equations
\beqa
A-\:\:\quad B &=& \tfrac{1}{2\tau} \\
G^-A+G^+B &=& 0
\eeqa
which for $\Id$ the identity has the unique solution 
\beqa
A=\frac{1}{2\tau} (G^-)^{-1} G^+(\Id+(G^-)^{-1} G^+)^{-1}\,,\quad
B=- \frac{1}{2\tau} (\Id+(G^-)^{-1} G^+)^{-1}\,.
\eeqa 
Finally, we can evaluate $\Phi(t)$ via integration of
\beqa
\dot \Phi(t) &=&  \frac{1}{2\tau}  \LMg_{\x(t)}^{-1} \! \left(\int_0^\tau \!\! \LMg_{\x(s)}^{-1} \d s\!\right)^{\!\!-1} \!\!\! \int_\tau^1 \!\!\LMg_{\x(s)}^{-1} \d s
\left(\Id\! +\! \left(\int_0^\tau \!\! \LMg_{\x(s)}^{-1} \d s\!\right)^{\!\!-1} \!\!\!\!\int_\tau^1 \!\!\LMg_{\x(s)}^{-1} \d s\! \right)^{\!\!-1}
\;\text{on }(0,\tau) \,,\\
\dot \Phi(t) &=& -  \frac{1}{2\tau}  \LMg_{\x(t)}^{-1} \! \left(\Id\! +\! \left(\int_0^\tau \!\! \LMg_{\x(s)}^{-1} \d s\!\right)^{\!\!-1}\!\!\!\! \int_\tau^1 \!\!\LMg_{\x(s)}^{-1} \d s \!\right)^{\!\!-1}
\;\text{on }(\tau,1) \,.
\eeqa
Let us remark that in the trivial case of the constant metric
$g_\x(v,v) = (v,v)_\V$ on a Hilbert space $\V$ (in which case geodesics are straight lines),
$\Phi$ is a piecewise affine function with $\Phi(0) =\Phi(1) =0$ and $\Phi(\tau)= \frac{1-\tau}2\Riesz$
for the Riesz isomorphism $\Riesz:\V'\to\V$.

Next, we define $\Psi_\tau$ as the Lagrangian interpolation $\Itau\Phi$ of $\Phi$ with $\Psi_\tau(k\tau) = \Phi(k\tau)$ for $k=0,\ldots, K$.
To show the claim, it remains to show the convergence of the right-hand side in \eqref{eq:LogError}.
At first we proceed as in the proof of Theorem\,\ref{thm:W12}.
In \eqref{eq:I-IX} we choose $\phi=\Phi\v$ and $\psi_\tau=\Psi_\tau\v$,
which turns the left-hand side into the desired error representation.
Hence, it remains to verify that the terms $I$ to $IX$ on the right-hand side of \eqref{eq:I-IX} vanish for $\tau \to 0$.
Indeed, using the same notation as in Theorem\,\ref{thm:W12} we obtain the following estimates:

\begin{itemize}
\setlength{\itemsep}{1ex}
\item[(I)] Due to the uniform boundedness of the metric, $|\,I\,|$ can be bounded by
$2 C^\ast \|\dot \x - \dot \x_\tau\|_{L^2((0,1);\V)} \|(\Phi-\Itau\Phi)\dot\ \v\|_{L^2((0,1);\V)}\,,$
and for $t\in((k-1)\tau,k\tau)$ one obtains\vspace{-1ex}
\begin{multline*}
\textstyle\|(\Phi-\Itau\Phi)\dot\ (t)\v\|_\V=\|\tfrac1\tau\int_{(k-1)\tau}^{k\tau}(\dot\Phi(t)-\dot\Phi(s))\v\,\d s\|_\V\\
\textstyle\leq\tfrac1\tau\int_{(k-1)\tau}^{k\tau}\|\LMg_{\x(t)}^{-1}-\LMg_{\x(s)}^{-1}\|\,\d s\,\|\LMg_{\x(\hat t)}\dot\Phi(\hat t)\v\|_\V
=O(\tau)\|\dot\Phi(\hat t)\v\|_\V
\end{multline*}
with $\hat t=0$ for $k=1$ and $\hat t=1$ else.
Using $\dot\Phi(0)=O(\tau^{-1})$ and $\dot\Phi(1)=O(1)$, we obtain $\|(\Phi-\Itau\Phi)\dot\ \v\|_{L^2((0,1);\V)}=O(\tau\|\v\|_{\V'})$.
\item[(II)] Due to \eqref{eq:contEL} the term $II$ vanishes as before.
\item[(III)] By \eqref{eq:discreteEL2} and this time
\eqref{eq:errLog} the term $|\,III\,|$ can be estimated by
$$C\tau\|\psi_\tau\|_{L^\infty((0,1);\V)}+C\tau^2\|w_1\|_\V+C\tau\sup_{k>1}\|w_k\|_\V\,,$$
where $\|w_1\|_\V \leq C \tau^{-1}\|\v\|_{\V'}$ and $\sup_{k>1}\|w_k\|_\V=O(\|\v\|_{\V'})$.\\
Using $\|\psi_\tau\|_{L^\infty((0,1);\V)}=O(1)$, $|\,III\,|$ is bounded by $C \tau\|\v\|_{\V'}$.
\item[(IV)] Due to the smoothness of the metric $g$, the quadrature error in $IV$ is bounded by\vspace{-1ex}
$$O(\tau^2\sum_k\|v_k\|_\X\|v_k\|_\V\|\dot\psi_\tau\|_\V)=O(\tau^2\|w_1\|_\V+\tau\sup_{k>1}\|w_k\|_\V)=O(\tau\|\v\|_{\V'})\,.$$ 
\vspace{-4ex}

\item[(V)] Using Theorem\,\ref{thm:equidistribution},
one obtains\vspace{-1ex}
\beqa
|\,V\,| &\leq&\textstyle C\int_0^1\|\x-\x_\tau\|_\X\|\dot\psi_\tau\|_\V\|\dot\x_\tau\|_\V\,\d t \\
&\leq& C\|\x-\x_\tau\|_{L^\infty((0,1);\X)}\|\dot\psi_\tau\|_{L^1((0,1);\V)}\|\dot\x_\tau\|_{L^\infty((0,1);\V)} \\
&=& O(\|\x-\x_\tau\|_{L^\infty((0,1);\X)}\|\v\|_{\V'})\,.
\eeqa
\item[(VI)] By the smoothness of $g$ we get
$|\,VI\,| \leq C\tau^2\!\sum_k\|\psi_{k\!-\!1}\|_\X\|v_k\|_\V^2\|v_k\|_\X=O(\tau\|\v\|_{\V'})\,.$
\item[(VII)] Again due to the smoothness of $g$ we achieve\vspace{-1ex}
$$|\,VII\,| \leq C\tau^2\|w_1\|_\X\|v_1\|_\V^2 + C\sum_{k>1}\tau^2\|w_k\|_\X\|v_k\|^2_\V \leq C \tau\|\v\|_{\V'}\,.$$
\vspace{-4ex}

\item[(VIII)] For the term $VIII$ one obtains the estimate\vspace{-1ex}
\beqa |\,VIII\,| &\leq& C\|\x-\x_\tau\|_{L^2((0,1);\X)}\|\psi_\tau\|_{L^2((0,1);\X)}\|\dot\x_\tau\|^2_{L^\infty((0,1);\V)}\\
&=&O(\|\x-\x_\tau\|_{L^2((0,1);\X)}\|\v\|_{\V'})\,.
\eeqa
\item[(IX)] Finally, for the last term we get\vspace{-1ex}
\beqa 
|\,IX\,| &\leq& 
C\|\psi_\tau\|_{L^2((0,1);\X)}(\|\dot\x\|_{L^\infty((0,1);\V)}+\|\dot\x_\tau\|_{L^\infty((0,1);\V)})
\|\dot\x-\dot\x_\tau\|_{L^2((0,1);\V)}\\
&=&O(\|\dot\x-\dot\x_\tau\|_{L^2((0,1);\V)}\|\v\|_{\V'})\,.
\eeqa
\end{itemize}
Collecting all these estimates we obtain
\beq
\left(\v,K\Log{K}_{\x}{\tilde \x} - \log_{\x}{\tilde \x}\right)_{\V',\V} 
\leq C\big(\tau+\|\dot\x-\dot\x_\tau\|_{L^2((0,1);\V)}
+\|\x-\x_\tau\|_{L^\infty((0,1);\X)}\big)\|\v\|_{\V'}\,.
\eeq
Using Theorems\,\ref{thm:Linfty} and \ref{thm:W12} the right hand side of this estimate converges to $0$ as $K\to\infty$ for all $\v\in \V'$, which proves the claim.
\end{proof}

Next we consider the existence and convergence of the discrete exponential.
This requires several preparations.

\begin{lemma}[Local uniqueness of $\Log{2}$]\label{thm:uniquenessLog2}
Under the hypotheses \ref{enm:smoothMetric} and \ref{enm:smoothEnergy}, there exists an $\epsilon>0$
such that for any $\x_0,\x_2\in\manifold$ with $\x_2\in B_\epsilon(\x_0)=\{\x\;|\;\|\x-\x_0\|_\V\leq\epsilon\}$,
$\E[(\x_0,\cdot,\x_2)]:\manifold\to\R$ is strictly convex with bounded coercive Hessian
on $B_\epsilon(\x_0)$, and $\Log{2}_{\x_0}(\x_2)$ is unique.
\end{lemma}

\begin{proof}
For given $\epsilon$, consider $\x_0,\x_2\in\manifold$ with $\|\x_2-\x_0\|_\V<\epsilon$.
The Hessian of $\E[(\x_0,\cdot,\x_2)]$ at some $\x_1$ in the ball $B_\epsilon(\x_0)$ is given
by $2(\W_{,22}[\x_0,\x_1]+\W_{,11}[\x_1,\x_2])=8g_{\x_0}+O(\epsilon)$ with uniform constants,
where we have used the smoothness of $\W$ and Lemma\,\ref{thm:consistent}.
For $\epsilon$ small enough, this is bounded and coercive (independent of $\x_0$), and thus $\E[(\x_0,\cdot,\x_2)]$ is strictly convex (cf. Theorem \ref{thm:uniqueness}).
This also implies uniqueness of $\Log{2}_{\x_0}(\x_2)$
(via Theorem\,\ref{thm:energyConvergenceRate} it is easy to see that the optimal $\x_1$ cannot lie outside $B_\epsilon(\x_0)$).
\end{proof}

\begin{remark}
The above result can be generalized to $K$-geodesics inside an $\epsilon$-ball.
It is a refinement of Theorem\,\ref{thm:uniqueness} in the sense that it provides a uniform uniqueness radius $\epsilon$.
Also, it implies the differentiability of $\Log{2}_{\x_0}(\x_2)$ with respect to $\x_2\in B_\epsilon(\x_0)$:
The Hessian $D^2\E[(\x_0,\cdot,\x_2)]=2(\W_{,22}[\x_0,\cdot]+\W_{,11}[\cdot,\x_2])$ is coercive so that by the implicit function theorem,
\vspace{-2ex}
\begin{align*}
D\Log{2}_{\x_0}(\x_2)
&=-\big(\W_{,22}[\x_0,\x_1]+\W_{,11}[\x_1,\x_2]\big)^{-1}\W_{,12}[\x_1,\x_2]\\
&=[4g_{\x_0}+O(\epsilon)]^{-1}(2g_{\x_0}+O(\epsilon))=\tfrac12\Id+O(\epsilon)
\end{align*}
for $\x_1:=\x_0+\Log{2}_{\x_0}(\x_2)$.
\end{remark}

\begin{lemma}\label{lem:Log2}
For $\x_2\!\in\!B_\epsilon(\x_0)$ the estimate $\x_0\!+\!\Log{2}_{\x_0}(\x_2)\!=\!\tfrac{\x_{0\!}+\x_2}2\!+\!O(\epsilon^{3/2})$ holds.
\end{lemma}
\begin{proof}
Indeed, using $\W[\x,\tilde\x]=\dist^2(\x,\tilde\x)+O(\dist^3(\x,\tilde\x))
=g_{\x_0}(\x-\tilde\x,\x-\tilde\x)+O(\epsilon^3)$ for $\x,\tilde\x\in B_\epsilon(\x_0)$ and
$2[\W[\x_0,\x_1(\x_2)]+\W[\x_1(\x_2),\x_2)]]=g_{\x_0}(\x_2-\x_0,\x_2-\x_0)+O(\epsilon^3)$
for $\x_1(\x_2)=\x_0+\Log{2}_{\x_0}(\x_2)$ (by Theorem\,\ref{thm:energyConvergenceRate}), we get\\[1ex]
$\hspace*{.5ex}2g_{\x_0}(\tfrac{\x_0+\x_2}2\!-\!\x_1(\x_2),\tfrac{\x_0+\x_2}2\!-\!\x_1(\x_2))$\\[1ex]
$\hspace*{2ex}=g_{\x_0}(\x_0\!-\!\x_1(\x_2),\x_0\!-\!\x_1(\x_2))
+g_{\x_0}(\x_2\!-\!\x_1(\x_2),\x_2\!-\!\x_1(\x_2))
\!-\!\tfrac12g_{\x_0}(\x_2\!-\!\x_0,\x_2\!-\!\x_0)$\\[1ex]
$\hspace*{2ex}=\W[\x_0,\x_1(\x_2)]\!+\!\W[\x_1(\x_2),\x_2]
-\left(\W[\x_0,\x_1(\x_2)]\!+\!\W[\x_1(\x_2),\x_2]\right)
+O(\epsilon^3)
=O(\epsilon^3)\,.$
\end{proof}
\notinclude{ 
\begin{lemma}[Monotonicity of $\Log{2}$]
Under the assumptions of Lemma\,\ref{thm:uniquenessLog2},
$\Log{2}_{\x_0}$ is a monotone operator on $B_\epsilon(\x_0)\subset\manifold$, i.\,e.
$$\big(\Log{2}_{\x_0}(\x_2)-\Log{2}_{\x_0}(\tilde\x_2),\x_2-\tilde\x_2\big)_\V>0\,.$$
\end{lemma}
\begin{proof}
$\big(\Log{2}_{\x_0}(\x_2)-\Log{2}_{\x_0}(\tilde\x_2),\x_2-\tilde\x_2\big)_\V
=\int_0^1\big(D\Log{2}_{\x_0}(\tilde\x_2+t(\x_2-\tilde\x_2))(\x_2-\tilde\x_2),\x_2-\tilde\x_2\big)_\V\,\d t
=(\tfrac12+O(\epsilon))\|\x_2-\tilde\x_2\|_\V^2$.
\end{proof}
}
\vspace{1ex}

\begin{lemma}[Local existence of $\Exp{2}$]\label{thm:existenceExp2}
Under the hypotheses \ref{enm:smoothMetric} and \ref{enm:smoothEnergy}, there exists an $\epsilon>0$
such that $\Exp{2}_{\x_0}(v)$ exists for any $\x_0\in\manifold$, $v\in\V$, with $\|v\|_\V<\epsilon$.
\end{lemma}

\begin{proof}
Choose $\epsilon$ to be one third the value from Lemma\,\ref{thm:uniquenessLog2}.
For $\x_0\in\manifold$, $v\in\V$ with $\|v\|_\V<\epsilon$, define the operator $T_v:B_{3\epsilon}(\x_0)\to\manifold$ by
$
T_v\x_2=\x_2+v-\Log{2}_{\x_0}(\x_2)\,.
$
Without loss of generality we may assume $\epsilon$ to be small enough
such that by Lemma\,\ref{lem:Log2}, $T_v:B_{3\epsilon}(\x_0)\to B_{3\epsilon}(\x_0)$.
Its derivative is given by $\Id-D\Log{2}_{\x_0}(\x_2)=\tfrac12\Id+O(\epsilon)$,
hence $T_v$ is a contraction, and by the Banach fixed point theorem has a unique fixed point $\Exp2_{\x_0}(v)$.
\end{proof}

\begin{remark}
By the inverse function theorem, $\Exp2_{\x_0}$ is differentiable at $v$ with $\|v\|_\V<\epsilon$, and
$
D\Exp2_{\x_0}(v)=[D\Log2_{\x_0}(\Exp2_{\x_0}(v))]^{-1}=2\Id+O(\epsilon)\,.
$
\end{remark}

\begin{lemma}
For $\|v\|_\V\leq\epsilon$ the estimate $\Exp{2}_{\x_0}(v)=(\x_0+2v)+O(\epsilon^{3/2})$ holds.
\end{lemma}
\begin{proof}
Let us abbreviate $\x_1=\x_0+v$ and $\x_2(\x_1)=\Exp{2}_{\x_0}(v)$.
Then as in Lemma\,\ref{lem:Log2},\\[.5ex]
$\hspace*{.5ex}\tfrac12g_{\x_0}(\x_0+2v\!-\!\x_2(\x_1),\x_0+2v\!-\!\x_2(\x_1))$\\[.7ex]
$\hspace*{2ex}=g_{\x_0}(\x_0\!-\!\x_1,\x_0\!-\!\x_1)
+g_{\x_0}(\x_2(\x_1)\!-\!\x_1,\x_2(\x_1)\!-\!\x_1)
\!-\!\tfrac12g_{\x_0}(\x_2(\x_1)\!-\!\x_0,\x_2(\x_1)\!-\!\x_0)$\\[.7ex]
$\hspace*{2ex}=\W[\x_0,\x_1]+\W[\x_1,\x_2(\x_1)]
\!-\!\left(\W[\x_0,\x_1]+\W[\x_1,\x_2(\x_1)]\right)
+O(\epsilon^3)
=O(\epsilon^3)\,.$
\end{proof}
\vspace{1ex}

\begin{remark}\label{rem:ApprOrderExp2Log2}
Using the above boundedness, our previous estimates can be further improved to
$\|\Exp{2}_{\x_0}(v)-(\x_0+2v)\|_\V=O(\|v\|_\V^2)$ and 
$\|\Log{2}_{\x_0}(\x_2)-(\tfrac{\x_0+\x_2}2-\x_0)\|_\V=O(\|\x_2-\x_0\|_\V^2)$
for $\|v\|_\V$ and $\|\x_2-\x_0\|_\V$ smaller than $\epsilon$.
Indeed, if $(\x_0,\x_1,\x_2)$ is a discrete geodesic with $v_1=v=\x_1-\x_0$ and $v_2=\x_2-\x_1$,
and if $\psi_1$ is a variation of $\x_1$, then the optimality condition implies
\begin{align*}
0&=\W_{,2}[\x_0,\x_1](\psi_1)+\W_{,1}[\x_1,\x_2](\psi_1)\\
&=\W_{,2}[\x_1,\x_1](\psi_1)-\W_{,21}[\x_1,\x_1](\psi_1,\x_1-\x_0)+O(\|\x_1-\x_0\|_\V^2\|\psi_1\|_\V)\\
&\quad+\W_{,1}[\x_1,\x_1](\psi_1)+\W_{,12}[\x_1,\x_1](\psi_1,\x_2-\x_1)+O(\|\x_2-\x_1\|_\V^2\|\psi_1\|_\V)\\
&=2g_{\x_1}(\psi_1,v_1-v_2)+O((\|v_1\|_\V^2+\|v_2\|_\V^2)\|\psi_1\|_\V)\,,
\end{align*}
where we used Lemma\,\ref{thm:consistent}.
Using the coercivity of $g_{\x_1}$, the desired estimates can now easily be derived.
\end{remark}

\begin{theorem}[Existence and convergence of $\Exp{K}$]\label{thm:ExpK}
Let $\x:[0,1]\to\manifold$ be a smooth geodesic. Under the hypotheses \ref{enm:smoothMetric} and \ref{enm:smoothEnergy},
$\Exp{K}_{\x(0)}(\frac{\dot\x(0)}K)$ exists for $K$ large enough, and for $\tau=\frac1K$ one obtains 
$\big\|\x(1)-\Exp{K}_{\x(0)}\big(\tfrac{\dot\x(0)}K\big)\big\|_\V=O(\tau)\,.$
\end{theorem}
\begin{proof}
To examine the convergence of the discrete exponential, we need to linearize the optimality condition\,\eqref{eq:discreteEL}.
Let $(\x_0,\ldots,\x_K)$ be some discrete path in $\manifold$ and let $v_k=(\x_k-\x_{k-1})/\tau$ for $\tau=1/K$.
As in the proof of Theorem\,\ref{thm:W12}, for $\psi_k\in\V$ we find
{\allowdisplaybreaks
\begin{align}
&\W_{,2}[\x_{k-1},\x_k](\psi_k)\!+\!\W_{,1}[\x_k,\x_{k\!+\!1}](\psi_k)\nonumber\\
=&\,\tau\!\!\int_0^\tau\!\!\!(1-\tfrac{t}\tau)\Big[
\W_{,222}[\x_{k-1},\x_{k-1}\!+\!tv_k](v_k,v_k,\tfrac{t}\tau\psi_k)
\!+\!2\W_{,22}[\x_{k-1},\x_{k-1}\!+\!tv_k](v_k,\tfrac1\tau\psi_k)\nonumber\\
&+\!\W_{,221}[\x_k,\x_k\!+\!tv_{k\!+\!1}](v_{k\!+\!1},v_{k\!+\!1},\psi_k)
\!+\!\W_{,222}[\x_k,\x_k\!+\!tv_{k\!+\!1}](v_{k\!+\!1},v_{k\!+\!1},\psi_k)\nonumber\\
&-\W_{,222}[\x_k,\x_k\!+\!tv_{k\!+\!1}](v_{k\!+\!1},v_{k\!+\!1},\tfrac{t}\tau\psi_k)
-2\W_{,22}[\x_k,\x_k\!+\!tv_{k\!+\!1}](v_{k\!+\!1},\tfrac1\tau\psi_k)
\Big]\,\d t\nonumber\\
=&\,\tau\!\!\int_0^\tau(1-\tfrac{t}\tau)\Big[
-\tfrac4\tau(g_{\x_k}(v_{k\!+\!1},\psi_k)-g_{\x_{k-1}}(v_{k},\psi_k))\!+\!2D_\x g_{\x_k}(\psi_k)(v_{k\!+\!1},v_{k\!+\!1})\nonumber\\
&\quad+ \int_0^1  \Big( 2 \W_{,222}[\x_{k\!-\!1},\x_{k\!-\!1}\!+\!r t v_k](v_k,\tfrac1\tau\psi_k,tv_k) \nonumber \\[-1ex]
&\quad\qquad \quad -\!2\W_{,222}[\x_k,\x_k\!+\!r t v_{k\!+\!1}](v_{k\!+\!1},\tfrac1\tau\psi_k,tv_{k\!+\!1})\nonumber\\
&\quad\qquad  \quad +\!\W_{,2212}[\x_k,\x_k\!+\!r t v_{k\!+\!1}](v_{k\!+\!1},v_{k\!+\!1},\psi_k,tv_{k\!+\!1}) \nonumber \\
&\quad\qquad \quad +\!\W_{,2222}[\x_k,\x_k\!+\!r t v_{k\!+\!1}](v_{k\!+\!1},v_{k\!+\!1},\psi_k,tv_{k\!+\!1}) \Big) \d r \nonumber\\
&\quad +\!\W_{,222}[\x_{k-1},\x_{k-1}\!+\!tv_k](v_k,v_k,\tfrac{t}\tau\psi_k) 
-\W_{,222}[\x_k,\x_k\!+\!tv_{k\!+\!1}](v_{k\!+\!1},v_{k\!+\!1},\tfrac{t}\tau\psi_k)
\Big]\,\d t\nonumber\\
=&-\tau^2\left[2\frac{g_{\x_k}(v_{k\!+\!1},\psi_k)-g_{\x_{k-1}}(v_{k},\psi_k)}\tau-D_\x g_{\x_k}(\psi_k)(v_{k\!+\!1},v_{k\!+\!1})\right]\nonumber\\
&+\!O\!\left[(\tau(\|v_k\|_\V^3\!\!+\!\!\|v_{k\!+\!1}\|_\V^3)\!+\!(\tau\|v_k\|_\V^2\!\!+\!\!\|v_{k\!+\!1}\!\!-\!v_k\|_\V)
(\|v_{k\!+\!1}\|_\V\!\!+\!\!\|v_k\|_\V))\tau^2\|\psi_k\|_\V\right]\hspace{-1.5ex}\label{eqn:linOptCond}
\end{align}
}
assuming sufficient differentiability of $\W$ and using the simple Taylor expansion formula
\begin{multline*}
f[a_2,b_2](c_2,c_2,d)-f[a_1,b_1](c_1,c_1,d)
=f[a_1,b_1](c_2-c_1,c_2,d)+f[a_1,b_1](c_1,c_2-c_1,d)\\
+f_{,1}[a,b](c_2,c_2,d,a_2-a_1)
+f_{,2}[a,b](c_2,c_2,d,b_2-b_1)
\end{multline*}
for a spatially dependent, differentiable trilinear form $f$ and $a=(1-\xi)a_1+\xi a_2$, $b=(1-\xi)b_1+\xi b_2$ for some $\xi\in[0,1]$.

Next, let us abbreviate $t_k=k\tau$.
The smooth geodesic $\x$ satisfies
\begin{align*}
0&=-2\tfrac\d{\d t}g_{\x(t)}(\dot\x(t),\psi_k)|_{t=t_k}+D_\x g_{\x(t_k)}(\psi_k)(\dot\x(t_k),\dot\x(t_k))\\
&=-2\frac{g_{\x(t_k)}(u_{k+1},\psi_k)-g_{\x(t_{k-1})}(u_k,\psi_k)}\tau+D_\x g_{\x(t_k)}(\psi_k)(u_{k+1},u_{k+1})+O(\tau\|\psi_k\|_\V)
\end{align*}
for $u_k=\tfrac{\x(t_k)-\x(t_{k-1})}\tau$ and $|O(\tau \|\psi_k\|_\V)| \leq C \tau(\|\dddot\x\|_\V+\|\ddot\x\|_\V(\|\dot\x\|_\V+\tau\|\ddot\x\|_\V))\|\psi_k\|_\V$.
Now assume $\x_{k+1}:=\Exp{k+1}_{\x(0)}(\dot \x(0)/K)$ exists for some $k<K$ (and thus also $\x_0,\ldots,\x_k$), then the left-hand side of \eqref{eqn:linOptCond} is zero.
Subtracting $1/\tau^2$ times \eqref{eqn:linOptCond} from the above equation yields
\begin{align}
&O\left[(\tau+\tau\|v_k\|_\V^3+\tau\|v_{k+1}\|_\V^3+(\tau\|v_k\|_\V^2+\|v_{k+1}-v_k\|_\V)(\|v_{k+1}\|_\V+\|v_k\|_\V))\|\psi_k\|_\V\right]\nonumber\\
&=2\frac{g_{\x_k}(v_{k+1},\psi_k)-g_{\x_{k-1}}(v_{k},\psi_k)}\tau
-2\frac{g_{\x(t_k)}(u_{k+1},\psi_k)-g_{\x(t_{k-1})}(u_k,\psi_k)}\tau\nonumber\\
&\quad-D_\x g_{\x_k}(\psi_k)(v_{k+1},v_{k+1})
+D_\x g_{\x(t_k)}(\psi_k)(u_{k+1},u_{k+1})\label{eqn:timeSteppingError}
\end{align}
Introducing $e_k=\x_k-\x(t_k)$ as well as $e_k^v=\frac{e_k-e_{k-1}}\tau$, the first line on the right-hand side can be rewritten as
\begin{flalign*}
&\mathrlap{2 \frac{g_{\x_k}(v_{k+1}-v_k,\psi_k)-g_{\x(t_k)}(u_{k+1}-u_k,\psi_k)}\tau}&\\
&&+2 D_\x g_{\x_k}(v_k)(v_k,\psi_k)- 2 D_\x g_{\x(t_k)}(u_k)(u_k,\psi_k)
+O(\tau(\|v_k\|_\V^3+\|u_k\|_\V^3)\|\psi_k\|_\V) \\
=&\mathrlap{ 2 g_{\x_k}\Big(\dfrac{e_{k+1}^v-e_k^v}\tau,\psi_k\Big) +O\big[(\|e_k\|_\V\|\tfrac{u_{k+1}-u_k}\tau\|_\V+\|e_k\|_\V\|u_k\|_\V^2}& \\
&&+\|e_k^v\|_\V(\|u_k\|_\V+\|v_k\|_\V)+\tau(\|v_k\|_\V^3+\|u_k\|_\V^3))\|\psi_k\|_\V\big]\,,
\end{flalign*}
and the second line is bounded by $O[(\|e_{k\!}\|_\V\|u_{k+1\!}\|_\V^2+\|e_{k+1\!}^v\|_\V\|u_{k+1\!}+v_{k+1\!}\|_\V)\|\psi_{k\!}\|_\V]$.
Thus, exploiting the boundedness of $\x$ and its derivatives as well as
\begin{align*}
u_k&=\dot\x(t_k)+\tau O(\ddot\x)\,,\quad
u_{k+1}-u_k=\tau\ddot\x(t_k)+\tau^2O(\dddot\x)\,,\\
\|v_k\|_\V&\leq\|u_k\|_\V+\|e_k^v\|_\V\,,\quad
\|v_{k+1}-v_k\|_\V\leq\|u_{k+1}-u_k\|_\V+\|e_{k+1}^v\|_\V+\|e_k^v\|_\V\,,
\end{align*}
\eqref{eqn:timeSteppingError} can be rewritten as
\begin{equation*}
\begin{pmatrix}\frac{e_{k+1}-e_k}\tau\\[1ex] \frac{e_{k+1}^v-e_k^v}\tau\end{pmatrix}
=\begin{pmatrix}
e_{k+1}^v\\[2ex] O[\tau+\|e_k\|_\V+(\|e_k^v\|_\V+\|e_{k+1}^v\|_\V)(1+\|e_k^v\|_\V+\|e_{k+1}^v\|_\V)^2]
\end{pmatrix}\,,
\end{equation*}
where we have also used the uniform coerciveness of $g$.
If $\|e_j^v\|_\V\leq1$ for $j=0,\ldots,k+1$, then the above implies the existence of a constant $C>0$
such that with $U=C\left(\begin{smallmatrix}1&1\\1&1\end{smallmatrix}\right)$,
\begin{equation*}
(\Id-\tau U)\begin{pmatrix}\|e_{j+1}\|_\V\\\|e_{j+1}^v\|_\V\end{pmatrix}
\leq(\Id+\tau U)\begin{pmatrix}\|e_{j}\|_\V\\\|e_{j}^v\|_\V\end{pmatrix}
+C\tau^2\,,\qquad j=0,\ldots,k,
\end{equation*}
which together with $e_0=0$, $\|e_1^v\|_\V=O(\tau)$ by classical arguments implies
$\|e_j\|_\V,\|e_j^v\|_\V\leq\kappa\tau$, $j=0,\ldots,k+1$, for some $\kappa>0$
which is independent of $k$ and $K$.

Now choose $\tfrac1\tau=K$ large enough such that $C_K:=\kappa\tau^2+\tau\max_{t\in[0,1]}\|\dot\x(t)\|_\V<\epsilon$
for $\epsilon\ll1$ from Lemma\,\ref{thm:existenceExp2}.
Then by induction, $\Exp{k}(\dot\x(0)/K)=\x_k$ exists with $\|e_k\|_\V,\|e_k^v\|_\V<\kappa\tau$.
Indeed, for $k=0$ the situation is clear.
Now if $\x_{j}$ exists with $\|e_j\|_\V,\|e_j^v\|_\V<\kappa\tau$ for all $j<k$,
then 
$
\|\x_{k-1}-\x_{k-2}\|_\V 
\leq  \left(\|u_{k-1}\|_\V  + \|e^v_{k-1}\|_\V\right) \tau 
\leq  \left(\max_{t\in[0,1]}\|\dot\x(t)\|_\V + \kappa \tau \right)  \tau \leq C_K < \epsilon
$
so that $\x_k=\Exp2_{\x_{k-2}}(\x_{k-1}-\x_{k-2})$ exists by Lemma\,\ref{thm:existenceExp2}
with $\x_k-\x_{k-1}=\x_{k-1}-\x_{k-2}+O(\epsilon^{3/2})$.
Thus $\|e_k\|_\V,\|e_k^v\|_\V<1$, which via the above error estimates in turn implies $\|e_k\|_\V,\|e_k^v\|_\V<\kappa\tau$.
\end{proof}

\begin{theorem}[Existence and convergence of $\ParTp_{\x_K,\ldots,\x_0}$]
Let $\x:[0,1]\to\manifold$ be a smooth path,
and let $\zeta:[0,1]\to\V$ be a parallel vector field along $\x$.
For $K\in\N$ and $\tau=\frac1K$, let $\x_k=\x(k\tau)$, then under the hypotheses \ref{enm:smoothMetric} and \ref{enm:smoothEnergy}, the discrete parallel transport fulfills
$
\|K\ParTp_{\x_K,\ldots,\x_0}(\tfrac{\zeta(0)}K)-\zeta(1)\|_\V =O(\tau)\,.
$
\end{theorem}

\begin{proof}
At first, we examine the residual resulting from the evaluation of the discrete counterpart of the parallel transport equation on the interpolated continuous vector field $\zeta$.
$\zeta$ is defined to be parallel if $\frac{D}{\d t} \zeta =0$ holds for the covariant derivative 
$\frac{D}{\d t}$ of $\zeta$ along the curve $x$, which using the definition of the covariant derivative turns into 
\beqn\label{eqn:contParTransp}
0= g_{\x(t)}((\tfrac{D}{\d t} \zeta)(t), \psi) = g_{\x(t)}(\dot \zeta(t), \psi) + g_{\x(t)}(\Gamma(\dot \x(t),\zeta(t)),\psi) 
\eeqn
for all test vectors $\psi\in\V$ and $t\in [0,1]$.
Here the Christoffel tensor $\Gamma: \V \times \V \to \V$ is defined by
$
g_{\x}(\Gamma(v,w),\psi) = \tfrac12 \left[
(D_\x g_{\x})(w)(\psi,v) + (D_\x g_{\x})(v)(\psi,w) - (D_\x g_{\x})(\psi)(v,w)
\right]
$
for all $v,w,\psi\in\V$.
Let us abbreviate $t_k=k\tau$, $\zeta_k=\zeta(t_k)$, and $v_k=\frac{\x_k-\x_{k-1}}\tau$.
Applying $g_\x(v,w) = \frac12\W_{,22}[\x,\x](v,w)$ (Lemma\,\ref{thm:consistent}) and
\eqref{eqn:metricDeriv} 
as well as the approximations $\dot\zeta(t_{k-1})=\frac{\zeta_k-\zeta_{k-1}}{\tau}+O(\tau\|\ddot\zeta\|_\V)$,
$\dot\x(t_{k-1})=v_k+O(\tau\|\ddot\x\|_\V)$, \eqref{eqn:contParTransp} turns into
\beqan
&&O\Big[\tau(\sup_{t\in[0,1]}\|\ddot\zeta(t)\|_\V+\sup_{t\in[0,1]}\|\ddot\x(t)\|_\V\sup_{t\in[0,1]}\|\zeta(t)\|_\V)\|\psi\|_\V\Big]\nonumber\\
&=& \W_{,22}[\x_{k-1},\x_{k-1}](\tfrac{\zeta_k-\zeta_{k-1}}{\tau},\psi) \nonumber\\
&& + \tfrac12\Big(
\W_{,221}[\x_{k-1},\x_{k-1}](\psi, v_k,\zeta_{k-1}) +
\W_{,222}[\x_{k-1},\x_{k-1}](\psi, v_k,\zeta_{k-1}) \nonumber\\
&& \qquad + \W_{,221}[\x_{k-1},\x_{k-1}](\psi,\zeta_{k-1} ,v_k) +
\W_{,222}[\x_{k-1},\x_{k-1}](\psi,\zeta_{k-1}, v_k) \nonumber\\
&& \qquad - \W_{,221}[\x_{k-1},\x_{k-1}](v_k,\zeta_{k-1},\psi) -
\W_{,222}[\x_{k-1},\x_{k-1}](v_k,\zeta_{k-1},\psi)
\Big) \nonumber\\
&=& \W_{,22}[\x_{k-1},\x_{k-1}](\tfrac{\zeta_k-\zeta_{k-1}}{\tau},\psi) \nonumber\\
&& + \tfrac12\Big(
-\W_{,211}[\x_{k-1},\x_{k-1}](\psi, v_k,\zeta_{k-1}) -
\W_{,212}[\x_{k-1},\x_{k-1}](\psi, v_k,\zeta_{k-1}) \nonumber\\
&& \qquad + \W_{,221}[\x_{k-1},\x_{k-1}](\psi,\zeta_{k-1},v_k) -
\W_{,221}[\x_{k-1},\x_{k-1}](v_k,\zeta_{k-1},\psi)
\Big) \nonumber\\
&=& \W_{,22}[\x_{k-1},\x_{k-1}](\tfrac{\zeta_k-\zeta_{k-1}}{\tau},\psi) \nonumber\\
&& - \tfrac12\Big(
\W_{,211}[\x_{k-1},\x_{k-1}](\psi, v_k,\zeta_{k-1}) +
\W_{,221}[\x_{k-1},\x_{k-1}](v_k,\zeta_{k-1},\psi)
\Big) \nonumber\\
&=& \W_{,22}[\x_{k-1},\x_{k-1}](\tfrac{\zeta_k-\zeta_{k-1}}{\tau},\psi) 
- \frac{
\W_{,112}[\x_{k-1},\x_{k-1}] + \W_{,221}[\x_{k-1},\x_{k-1}](\zeta_{k-1},v_k,\psi)}{2}\,,\nonumber
\eeqan
where we have used Schwarz's theorem and the identity
\begin{multline*}
\W_{,221}[\x,\x]+\W_{,222}[\x,\x]=\W_{,111}[\x,\x]+\W_{,112}[\x,\x]\\
=-\W_{,121}[\x,\x]-\W_{,122}[\x,\x]=-\W_{,211}[\x,\x]-\W_{,212}[\x,\x]\,,
\end{multline*}
which can be derived by differentiating the identity 
$\W_{,22}[\x,\x]=\W_{,11}[\x,\x]=-\W_{,12}[\x,\x]=-\W_{,21}[\x,\x]$
from Lemma\,\ref{thm:consistent} with respect to $\x$.
Due to the smoothness of $\zeta$ and $\x$ we achieve the estimate 
\beqan
&&\W_{,22}[\x_{k-1},\x_{k-1}](\tfrac{\zeta_k-\zeta_{k-1}}{\tau},\psi)  \nonumber\\
&& - \tfrac12\Big(
\W_{,112}[\x_{k-1},\x_{k-1}] + \W_{,221}[\x_{k-1},\x_{k-1}]\Big)(\zeta_{k-1},v_k,\psi) = O(\tau \|\psi\|_\V)
\label{eqn:ParTranspTruncErr}
\eeqan

Next, we derive an estimate for the residual obtained when evaluating the discrete counterpart of the parallel transport equation on the discrete solution.
Thus, let us abbreviate $\xi_k=\ParTp_{\x_k,\ldots,\x_0}(\zeta(0)/K)$
and denote the center of the $k$\textsuperscript{th} parallelogram in Schild's ladder by $\x_k^c$.
$\xi_k$ and $\x_k^c$ satisfy for all $\psi\in\V$ the two nonlinear equations
\beqan
\W_{,2}[\x_{k-1}+\xi_{k-1},\x_k^c](\psi) + \W_{,1}[\x_k^c,\x_{k}](\psi) &=& 0\,, \label{eq:ELP1} \\
\W_{,2}[\x_{k-1}, \x_k^c](\psi) + \W_{,1}[\x_k^c,\x_k+\xi_k](\psi) &=& 0\,, \label{eq:ELP2}
\eeqan
which for the moment we assume to be uniquely solvable.
Upon Taylor expansion about $(\x_k^c,\x_k^c)$ analogously to \eqref{eqn:TaylorExpansion},
one obtains
\beqa
\W[\x_{k-1}\!+\!\xi_{k-1}, \x^c_{k}] &=&
\int_0^1 (1-s)\W_{,11}[\x^c_k\!+s (\x_{k-1}\!+\!\xi_{k-1}\!-\!\x^c_{k}),\x^c_k]
\\[-1ex]
&& \qquad \qquad \qquad (\x_{k-1}\!+\!\xi_{k-1}\!-\!\x^c_{k},\x_{k-1}\!+\!\xi_{k-1}\!-\!\x^c_{k})\, \d s\,, \\
\W[\x^c_{k}, \x_{k}] &=&
\int_0^1 (1-s)\W_{,22}[\x^c_{k},\x^c_{k}\! +\! s (\x_{k}\!-\!\x^c_{k})] (\x_{k}\!-\!\x^c_{k},\x_{k}\!-\!\x^c_{k})\, \d s\,.
\eeqa
Differentiating these expressions, the first equation \eqref{eq:ELP1} defining the discrete parallel transport  turns into
{
\begin{align*}
0&=\int_0^1(1-s)\Big[
-2\W_{,11}[\x_k^c,\x_k^c](\psi,\x_{k-1}\!+\!\xi_{k-1}-\x_k^c)\\[-.5ex]
&\qquad\quad-2 \int_0^1 \!\!\!\W_{,111}[\x_k^c\!+\!rs(\x_{k\!-\!1\!}\!+\!\xi_{k\!-\!1\!}\!-\!\x_k^c),\x_k^c](\psi,\x_{k\!-\!1\!}\!+\!\xi_{k\!-\!1\!}\!-\!\x_k^c,s(\x_{k\!-\!1\!}\!+\!\xi_{k\!-\!1\!}\!-\!\x_k^c)) \d r\\
&\qquad\!+\!\W_{,111}[\x_k^c\!+\!s(\x_{k-1}\!+\!\xi_{k-1}\!-\!\x_k^c),\x_k^c](\x_{k-1}\!+\!\xi_{k-1}\!-\!\x_k^c,\x_{k-1}\!+\!\xi_{k-1}\!-\!\x_k^c,(1-s)\psi)\\
&\qquad\!+\!\W_{,112}[\x_k^c\!+\!s(\x_{k-1}\!+\!\xi_{k-1}\!-\!\x_k^c),\x_k^c](\x_{k-1}\!+\!\xi_{k-1}\!-\!\x_k^c,\x_{k-1}\!+\!\xi_{k-1}\!-\!\x_k^c,\psi)\\
&\qquad-2\W_{,22}[\x_k^c,\x_k^c](\psi,\x_{k}-\x_k^c)\\
&\qquad\quad-2\int_0^1 \!\!\! \W_{,222}[\x_k^c,\x_k^c\!+\!rs(\x_{k}-\x_k^c)](\psi,\x_{k}-\x_k^c,s(\x_{k}-\x_k^c)) \d r\\
&\qquad\!+\!\W_{,221}[\x_k^c,\x_k^c\!+\!s(\x_{k}-\x_k^c)](\x_{k}-\x_k^c,\x_{k}-\x_k^c,\psi)\\
&\qquad\!+\!\W_{,222}[\x_k^c,\x_k^c\!+\!s(\x_{k}-\x_k^c)](\x_{k}-\x_k^c,\x_{k}-\x_k^c,(1-s)\psi)
\Big]\,\d s\\
&=-\W_{,11}[\x_k^c,\x_k^c](\psi,\x_{k-1}\!+\!\xi_{k-1}-\x_k^c)-\W_{,22}[\x_k^c,\x_k^c](\psi,\x_{k}-\x_k^c)\\
&\quad \!+\!\tfrac12\W_{,112}[\x_k^c,\x_k^c](\x_{k-1}\!+\!\xi_{k-1}-\x_k^c,\x_{k-1}\!+\!\xi_{k-1}-\x_k^c,\psi) \\
&\quad  \!+\!\tfrac12\W_{,221}[\x_k^c,\x_k^c](\x_{k}-\x_k^c,\x_{k}-\x_k^c,\psi)\\
&\quad  \!+\!O\left[(\|\x_{k-1}\!+\!\xi_{k-1}-\x_k^c\|_\V^3\!+\!\|\x_k-\x_k^c\|_\V^3)\|\psi\|_\V\right]\,.
\end{align*}
}
Here we used that for a spatially differentiable trilinear form $f$ the relation
$
-2f[a,a+rsv](s\psi,v,v)+f[a,a+sv]((1-s)\psi,v,v)
=f[a,a]((1-3s)\psi,v,v) + \,O(\|v\|^3\|\psi\|)
$
holds for $0\leq r,s \leq 1$, which after multiplication with $(1-s)$ integrates up to $O(\|v\|^3\|\psi\|)$.
Likewise, the second equation \eqref{eq:ELP2} turns into
\beqa
0&=&-\W_{,11}[\x_k^c,\x_k^c](\psi,\x_{k-1}-\x_k^c)-\W_{,22}[\x_k^c,\x_k^c](\psi,\x_{k}\!+\!\xi_k-\x_k^c)\\
&&+\tfrac12 \W_{,112}[\x_k^c,\x_k^c](\x_{k-1}\!-\!\x_k^c,\x_{k-1}\!-\!\x_k^c,\psi)\\
&&+\tfrac12\W_{,221}[\x_k^c,\x_k^c](\x_{k}\!+\!\xi_k-\x_k^c,\x_{k}\!+\!\xi_k-\x_k^c,\psi)\\
&&+O\left[(\|\x_{k-1}-\x_k^c\|_\V^3\!+\!\|\x_k\!+\!\xi_k-\x_k^c\|_\V^3)\|\psi\|_\V\right]
\eeqa
Subtracting the second from the first equation, dividing by $\tau^2$, and using $\W_{,11}=\W_{,22}$, one obtains
\begin{align}
0&=\W_{,22}[\x_k^c,\x_k^c](\tfrac{\frac1\tau\xi_k-\frac1\tau\xi_{k-1}}\tau,\psi) \nonumber\\
&\!+\!\tfrac12\Big[\W_{,112}[\x_k^c,\x_k^c](\tfrac{\xi_{k-1}}\tau,\tfrac{\xi_{k-1}\!+\!2(\x_{k-1}-\x_k^c)}\tau,\psi)
-\W_{,221}[\x_k^c,\x_k^c](\tfrac{\xi_{k}}\tau,\tfrac{\xi_{k}\!+\!2(\x_{k}-\x_k^c)}\tau,\psi)\Big]\nonumber \\
&\!+\!O\left[(\|\x_{k-1}-\x_k^c\|_\V^3\!+\!\|\x_k-\x_k^c\|_\V^3\!+\!\|\xi_{k-1}\|_\V^3\!+\!\|\xi_{k}\|_\V^3)\|\psi\|_\V\right]\,. \label{eq:resPdiscrete}
\end{align}
Finally, we derive an equation for the error propagation and subtract \eqref{eqn:ParTranspTruncErr} from \eqref{eq:resPdiscrete}.
Introducing the error $e_k=K\xi_k-\zeta_k = \frac{\xi_k}{\tau}-\zeta_k $, we arrive at
\begin{align}
&\W_{,22}[\x_k^c,\x_k^c](\tfrac{e_k-e_{k-1}}\tau,\psi)
=\tfrac12\big(\W_{,112}[\x_k^c,\x_k^c]+\W_{,221}[\x_k^c,\x_k^c]\big)(e_{k-1},v_k,\psi)\nonumber\\
&+O\Big[\big[
\tau+\|\x_{k-1}\!-\!\x_k^c\|_\V^3+\|\x_k\!-\!\x_k^c\|_\V^3+\|\xi_{k-1}\|_\V^3+\|\xi_{k}\|_\V^3
+\|\x_k^c\!-\!\x_{k-1}\|_\V\|\tfrac{\zeta_k-\zeta_{k-1}}\tau\|_\V\nonumber\\
&+\|\x_k^c-\x_{k-1}\|_\V\|\zeta_{k-1}\|_\V\|v_k\|_\V
+\tfrac1{\tau^2}\|\xi_{k-1}\|_\V\|\xi_{k-1}-2(\x_k^c-\tfrac{\x_k+\x_{k-1}}2)\|_\V\nonumber\\
&+\tfrac1\tau\|\xi_{k}-\xi_{k-1}\|_\V\|v_k\|_\V
+\tfrac1{\tau^2}\|\xi_k\|_\V\|\xi_k-2(\x_k^c-\tfrac{\x_k+\x_{k-1}}2)\|_\V
\big]\|\psi\|_\V\Big]\,.\label{eqn:parTranspDiff}
\end{align}
Applying the boundedness of $v_k$, $\zeta_k$, $\dot\zeta$, and the uniform estimates
\begin{align*}
\|\tfrac{\xi_k}\tau-\tfrac{\xi_{k-1}}\tau\|_\V&\leq\|\zeta_k-\zeta_{k-1}\|_\V+\|e_k\|_\V+\|e_{k-1}\|_\V\\
&=O(\tau+\|e_k\|_\V+\|e_{k-1}\|_\V)\\
\|\tfrac{\xi_k}\tau\|_\V&\leq\|\zeta_k\|_\V+\|e_k\|_\V\\
\|\x_k^c-\x_k\|_\V&=O(\|\x_{k-1}+\xi_{k-1}-\x_k\|_\V)\\
&=\tau O(\|v_k\|_K+\|\zeta_{k-1}\|_\V+\|e_{k-1}\|_\V)\\
\|\x_k^c-\x_{k-1}\|_\V&=O(\|\x_k+\xi_k-\x_{k-1}\|_\V) \\
&=\tau O(\|v_k\|_\V+\|\zeta_k\|_\V+\|e_k\|_\V)\\
\|\xi_{k-1}-2(\x_k^c-\tfrac{\x_k+\x_{k-1}}2)\|_\V&=O(\|\x_{k-1}+\xi_{k-1}-\x_k\|_\V^2)\\
&=\tau^2O((\|v_k\|_\V+\|\zeta_{k-1}\|_\V+\|e_{k-1}\|_\V)^2)\\
\|\xi_k-2(\x_k^c-\tfrac{\x_k+\x_{k-1}}2)\|_\V&=O(\|\x_{k-1}-\x_k-\xi_{k}\|_\V^2)\\
&=\tau^2O((\|v_k\|_K+\|\zeta_k\|_\V+\|e_k\|_\V)^2)
\end{align*}
(the last four follow from Remark\,\ref{rem:ApprOrderExp2Log2}, assuming $\|\xi_{k-1}\|_\V,\|\xi_k\|_\V$ to be sufficiently small), we find
\begin{equation}\label{eqn:parTranspErrorODE}
\tfrac{e_k-e_{k-1}}\tau=O[\|e_{k-1}\|_\V+\tau(1+\|e_{k-1}\|_\V)^3+\|e_k\|_\V+\tau(1+\|e_k\|_\V)^3]\,.
\end{equation}
As in the proof of Theorem\,\ref{thm:ExpK}, this implies the existence of a constant $C>0$ (independent of $K$) such that
if $\|e_j\|_\V\leq1$ for $j\leq k\leq K$, then $\|e_k\|_\V\leq C\tau$.

We now inductively show the existence of $\xi_k$ with $\|e_k\|_\V\leq C\tau$.
Choose $K$ large enough such that $C+\sup_{k=1,\ldots,K}\|\zeta_{k-1}\|_\V+\|v_k\|_\V+\|\zeta_k\|_\V<\epsilon K$
for a given $\epsilon\ll1$ from Lemma\,\ref{thm:existenceExp2}.
Clearly, $\xi_0$ exists, and $e_0=0$.
Now assume the existence of $\xi_j$, $j<k$, with $\|e_j\|_\V\leq C\tau$.
In particular, this implies $\|\xi_{k-1}\|_\V+\tau\|v_k\|_\V<\epsilon$
and thus by Lemmata\,\ref{thm:uniquenessLog2} and \ref{thm:existenceExp2} the existence of $\x_k^c$ and $\xi_k$.
Furthermore, Remark\,\ref{rem:ApprOrderExp2Log2} ensures $\|e_k\|_\V\leq1$ such that the above estimates imply $\|e_k\|_\V\leq C\tau$.
\end{proof}

The parallel transport also converges if the interpolated smooth path $(\x_0,\ldots,\x_K)$ is replaced by another approximating path,
for instance a discrete geodesic.

\begin{corollary}
Let $\x:[0,1]\to\manifold$ be a smooth path
and $\zeta:[0,1]\to\V$ a parallel vector field along $\x$.
For $K\in\N$ and $\tau=\frac1K$, let $\x_k$ satisfy $\|\x_k-\x(k\tau)\|_\V\leq\epsilon$, $k=0,\ldots,K$,
then under the hypotheses \ref{enm:smoothMetric} and \ref{enm:smoothEnergy},
$
\|K\ParTp_{\x_K,\ldots,\x_0}(\tfrac{\zeta(0)}K)-\zeta(1)\|_\V =O(\tau+\epsilon)\,.
$
\end{corollary}

\begin{proof}
Following the previous proof we arrive at a slightly altered version of \eqref{eqn:parTranspDiff},
\begin{align*}
&\W_{,22}[\x_k^c,\x_k^c](\tfrac{e_k-e_{k-1}}\tau,\psi)
=\tfrac12\big(\W_{,112}[\x_k^c,\x_k^c]+\W_{,221}[\x_k^c,\x_k^c]\big)(e_{k-1},v_k,\psi)\\
&+O\Big[\big[
\tau\!+\!\|\x_{k-1}\!-\!\x_k^c\|_\V^3+\|\x_k\!-\!\x_k^c\|_\V^3\!+\!\|\xi_{k-1}\|_\V^3+\|\xi_{k}\|_\V^3
\!+\!\|\x_k^c\!-\!\x(t_{k-1})\|_\V\|\tfrac{\zeta_k-\zeta_{k-1}}\tau\|_\V\\
&+\|\x_k^c-\x(t_{k-1})\|_\V\|\zeta_{k-1}\|_\V\|v_k\|_\V
+\tfrac{1}{\tau^2}\|\xi_{k-1}\|_\V\|\xi_{k-1}-2(\x_k^c-\tfrac{\x_k+\x_{k-1}}2)\|_\V\\
&+\tfrac{1}{\tau}\|\xi_{k}-\xi_{k-1}\|_\V\|v_k\|_\V
+\tfrac{1}{\tau^2}\|\xi_k\|_\V\|\xi_k-2(\x_k^c-\tfrac{\x_k+\x_{k-1}}2)\|_\V
\big]\|\psi\|_\V\Big]\,,
\end{align*}
where $t_k=k\tau$ and $v_k=\tfrac{\x(t_k)-\x(t_{k-1})}\tau$.
Using $\|\x_k^c-\x(t_{k-1})\|_\V\leq\|\x_k^c-\x_{k-1}\|_\V+\epsilon$ as well as the estimates from the previous proof
we obtain \eqref{eqn:parTranspErrorODE} with $O(\epsilon)$ added to the right-hand side,
from which the claim follows by the same arguments.
\end{proof}

\begin{corollary}
Let $\theta \in \V$ and $\eta: \x_A + \V \to \V$ be a smooth vector field.
Define $\eta^\tau=(\eta(\x),\eta(\x+\tau\theta))$,
then under the hypotheses \ref{enm:smoothMetric} and \ref{enm:smoothEnergy} we have
$
\|\tfrac1{\tau^2}\Nabla_{\tau\theta}(\tau\eta^\tau)-\nabla_\theta\eta\|_\V= O(\tau)\,.
$
\end{corollary}

\begin{proof}
Consider the continuous geodesic $\x(t)$ with $\x(0)=\x$ and $\dot\x(0)=\theta$.
We have $\|\x(\tau)-(\x+\tau\theta)\|_\V=O(\tau^2)$.
Now define $\xi_1=\tau\eta(\x+\tau\theta)$ and $\xi_0=\ParTp_{\x+\tau\theta,\x}^{-1}\xi_1$
as well as $\zeta_1=\eta(\x(\tau))$ and $\zeta_0$ as the vector $\zeta_1$ parallel transported from $\x(\tau)$ to $\x(0)$.
Furthermore introduce the error $e_k=\tfrac1\tau\xi_k-\zeta_k$, $k=0,1$.
Since $\ParTp_{\x+\tau\theta,\x}^{-1}$ is defined via the same discrete Euler--Lagrange equations 
(\eqref{eq:ELParallel1} and \eqref{eq:ELParallel2} for $k=0$, $\x_0= \x$, and $\x_1= \x+\tau\theta$)
as $\ParTp_{\x+\tau\theta,\x}$,
\eqref{eqn:parTranspErrorODE} holds also here,
only with $O(\|\x(\tau)-(\x+\tau\theta)\|_\V)=O(\tau^2)$ added to the right-hand side (exactly as in the previous corollary).
Using $e_1=O(\tau^2)$, we immediately see $e_0=O(\tau^2)$.
The relation $\nabla_\theta\eta=\tfrac{\zeta_0-\eta(\x)}\tau+O(\tau)$
together with $\zeta_0+O(\tau^2)=\tfrac1\tau\xi_0=\eta(\x)+\tfrac1\tau\Nabla_{\tau\theta}(\tau\eta^\tau)$ now implies the claim.
\end{proof}

\section{Applications}\label{sec:application}
In this section we give two applications of the presented convergence theory.
In general, it can be utilized in a straightforward fashion for finite-dimensional manifolds and for those infinite-dimensional manifolds where motion paths are compact in state space and the metric can be rephrased properly on Lagrangian velocity fields. In what follows we exemplify this in the case of the discrete geodesic calculus on embedded manifolds from Section \ref{sec:embeddedManifold} and on a space of viscous rods.

\subsection{Embedded finite-dimensional manifolds}
The most basic example has been given in Section\,\ref{sec:embeddedManifold} by
a smooth, complete, $(m-1)$-dimensional manifold $\manifold$, embedded in $\R^m$.
Due to the assumptions on $\manifold$, a continuous geodesic exists between any two points $\x_A,\x_B\in\manifold$.
Furthermore, we assume the embedding to be such that \eqref{eqn:coercivity} holds for the function $\W$
which represents the squared extrinsic distance between two points on $\manifold$.

Even though in this example $\manifold$ is not an affine space,
the existence proof for discrete geodesics (Theorem\,\ref{thm:existence})
still applies if in it we replace $\X$ by $\manifold$
(note that due to the finite dimensions, the proof no longer has to rely on a compact embedding of $\manifold$ into some larger space).
All subsequent results either do not make use of the manifold structure
or are essentially local in the sense that their proofs only consider a neighborhood of a continuous geodesic.
This characteristic enables us to transfer the results onto the example at hand:
Via a chart we locally identify the manifold $\manifold$ with an open subset of $\R^{m-1}$ and set $\X=\V=\R^{m-1}$.
In more detail, all convergence proofs involve a continuous geodesic or smooth path (the limit object of convergence),
around which a small compact tubular neighborhood $U$ can be identified with a compact subset $\phi(U)$ of $\R^{m-1}$ (where $\phi:U\to\R^{m-1}$ is the parametrization). Considering $\phi(U)\subset\X$ instead of $\manifold$ we are in the setting of Sections\,\ref{sec:gamma} and \ref{sec:convergence},
and all convergence results apply.

\subsection{A space of viscous rods}\label{sec:discreteshells}
Here we deal with two-dimensional shapes or closed curves $\shell \subset \R^2$
which are interpreted as rods of some small thickness $\delta>0$.
If (real physical) rods are plastically deformed,
physical energy is dissipated via internal material friction.
This viscous friction is predominantly caused by two mechanisms, namely
friction due to tangential (transversally uniform) in-plane deformation and friction due to (transversally non-uniform) deformation caused by bending.
The squared geodesic distance between two rods $\shell$ and $\tilde \shell$ can be defined as the minimum energy dissipated during a distortion of $\shell$ into $\tilde \shell$.
Approximations of this dissipated energy due to tangential distortion and bending are given by the following two functionals (\cf \cite{LeRa95,FrJaMo03}),
\beqa
\W_\tgl [\x,\tilde \x] =  \delta \int_0^1 W\left(\frac{|\tilde \x_{,s}|^2}{|\x_{,s}|^2}\right) |\x_{,s}| \d s \,,\quad
\W_\bnd[\x,\tilde\x] = \delta^3 \int_0^1 (\kappa[\tilde\x]-\kappa[\x])^2 |\x_{,s}| \d s\,,
\eeqa
where $W(\cdot)$ is a convex function acting on the pointwise tangential strain and attaining its minimum at $1$.
Furthermore,
$\kappa[\x] = \frac{1}{|\x_{,s}|} \left(\frac{\x_{,s}}{|\x_{,s}|}\right)_{\!\!,s} \cdot \frac{D^{90}\x_{,s}}{|\x_{,s}|}$
is the curvature of a curve $\x$, subscript $,\!s$ denotes the derivative, and $D^{90}$ counterclockwise rotation by $\tfrac\pi2$, \ie $n[\x] = \frac{D^{90} \x_{,s}}{|\x_{,s}|}$ is the normal on the curve $\x$.
The underlying Riemannian metric $g_{\x}$ on the space of rods satisfies
$\frac12 \W_{,22}[\x,\x] = g_{\x}$ with $\W = \W_\tgl  + \W_\bnd$ and is given by
\beqa
g_\x(v,v) &=&  \int_0^1 2 \delta \frac{1}{|\x_{,s}|} W''(1) |v^\tgl_{,s}|^2 +  \delta^3 \left(\partial_{\x} \kappa[\x](v)\right)^2 |\x_{,s}| \d s\,,
\eeqa
where $v^\tgl_{,s} = v_{,s} \cdot \frac{\x_{,s}}{|\x_{,s}|}$ denotes the tangential component of $v_{,s}$ and
\beqa
\partial_{\x} \kappa[\x](v) &=& - \frac{\kappa[\x]}{|\x_{,s}|} v^\tgl_{,s} + 
\frac{1}{|\x_{,s}|} \left(\frac{1}{|\x_{,s}|} 
\left(\Id - \frac{\x_{,s}}{|\x_{,s}|} \otimes \frac{\x_{,s}}{|\x_{,s}|}\right) v_{,s} \right)_{\!\!,s} \!\!\! \cdot n[\x]\,.
\eeqa
Linearization of the bending energy and the choice $W(A)\!=\!\frac{(1\!-\!A)^2}2$ lead to a simplified model
\beqn\label{eq:simpleshellW}
\hat \W[\x,\tilde \x] = \int_0^1 \frac{\delta}{2} \left(1-\frac{|\tilde \x_{,s}|^2}{|\x_{,s}|^2}\right)^2 |\x_{,s}| + \delta^3  (\tilde \x_{,ss}-\x_{,ss})^2 |\x_{,s}| \d s\,,
\eeqn
which corresponds via $\hat g_\x = \frac12 \hat\W_{,22}[\x,\x]$ to the simplified metric
\beqn\label{eq:simpleshellg}
\hat g_{\x}(v,v) = \int_0^1  2 \delta  \frac{|v^\tgl_{,s}|^2}{|\x_{,s}|} +    \delta^3 |v_{,ss}|^2 |\x_{,s}| \d s\,.
\eeqn
We will now show that our convergence theory applies to such a shape manifold.
To this end, we consider for $0\leq \alpha < \frac12$  the space
$\X = \{ \x \in C^{1,\alpha}(\R;\R^2) \,|\, \x(s+1) = \x(s)\}$
of $C^{1,\alpha}$-smooth rod curves represented by periodic parametrizations.
Finally, let $\V= \{ v \in W^{2,2}_{\mathrm{loc}}(\R;\R^2)\,|\, v(s+1) = v(s),\, \int_0^1 v(s) \d s = 0\,,
\int_0^1 v(s)\wedge \x_A(s) \d s =0\}$ for some fixed $\x_A \in \X$ be the space of velocity fields with zero average speed and angular momentum. Furthermore, we assume that $|\x_{A,s}| \geq c > 0$. Then, we identify the manifold $\manifold$ of rods locally with an open subset of $\x_A+\V$.
By Sobolev embedding results $\V$ is compactly embedded in $\X$, and $\V$ is a reflexive, separable space. Given $v \in L^2((0,1);\V)$ with $\int_0^1 \|v(t,\cdot)\|^2_{\W^{2,2}((0,1),\R^2)}\,\d t \leq \eta$ for sufficiently small $\eta$ we obtain from
$\x(t,s)= \x_A(s) + \int_0^t v(r,s) \d r$ that $\frac{c}{2} \leq |\x_{,s}(t,s)| \leq C$ for the path $\x\in L^2((0,1);\manifold)$ with $v=\dot\x$.
Furthermore, for $g_\x$ defined in \eqref{eq:simpleshellg} 
$g_{\x}(v,v) =0$ implies $v=0$.
From this we deduce that in a sufficiently small neighborhood of $\x_A$ 
the metric $g$ is uniformly bounded and uniformly coercive on $\V$ in the sense of \eqref{eq:boundg} and that $\x \mapsto g_\x$ 
and $\W$ defined in \eqref{eq:simpleshellW} are smooth.  Furthermore, the coercivity assumption 
\eqref{eqn:coercivity} for $\W$ is fulfilled.
Hence, for $\x_B - \x_A$ sufficiently small the direct method in the calculus of variations used in the proofs of Theorem\,\ref{thm:existenceContinuous} and Theorem\,\ref{thm:existence} can be applied to establish the existence of a continuous and a discrete $K$-geodesic.  Furthermore, Theorem\,\ref{thm:equidistribution} implies that any $K$-geodesic between $\x_A$ and $\x_B$
stays inside the same bounded region in $\x_A + \V$
so that the convergence analysis of the subsequent theorems can be restricted to such a neighborhood of $\x_A$.


\section*{Acknowledgements}
The authors acknowledge support of the Hausdorff Center for Mathematics funded by the German Science foundation.


{\small
\bibliographystyle{siam}
\bibliography{all,library,own}
}

\end{document}